\documentclass{amsart}
\usepackage[arrow, matrix, curve]{xy}
\usepackage{graphicx}
\usepackage[utf8]{inputenc}
\usepackage{longtable}
\usepackage{psboxit}
\usepackage{vmargin,hhline,pifont}
\usepackage{enumerate,fancyhdr}
\usepackage{amsmath}
\usepackage{amsthm}
\usepackage{amsfonts}
\usepackage{amssymb}
\usepackage{chapterbib}
\usepackage{marvosym}
\usepackage{verbatim}

\newtheorem{theorem}{Theorem}[section]
\newtheorem{proposition}[theorem]{Proposition}
\newtheorem{lemma}[theorem]{Lemma}

\theoremstyle{definition}
\newtheorem{definition}[theorem]{Definition}
\newtheorem{example}[theorem]{Example}

\newtheorem{notation}{Notation}[section]
\newtheorem{corollary}[theorem]{Corollary}
\theoremstyle{remark}
\newtheorem{remark}[theorem]{Remark}

\newcommand{\ndash}{\nobreakdash-\hspace{0pt}}

\DeclareMathOperator{\Mexp}{Mexp}
\DeclareMathOperator{\Id}{Id}

\makeindex

\begin{document}

\title{Holomorphic completions of affine Kac-Moody groups}
\author{Walter Freyn}

\begin{abstract}
We construct holomorphic loop groups and their associated affine Kac-Moody groups and prove that they are tame Fr\'echet manifolds. These results form the functional analytic basis for the theory of affine Kac-Moody symmetric spaces, presented first in the authors thesis. Our approach also solves completely the problem of complexification of loop groups; it allows a complete description of complex Kac-Moody groups and their non-compact real forms.
\end{abstract}

\maketitle

%\setcounter{page}{4}

%\tableofcontents

\section{Introduction}
This is the first one of a series of $5$ papers devoted to the construction and investigation of \emph{affine Kac-Moody symmetric spaces}. The existence of \emph{affine Kac-Moody symmetric spaces} was originally conjectured by Chuu-Lian Terng in her seminal paper~\cite{Terng95} noting severe technical obstacles on the way to their construction. The challenge to construct Kac-Moody symmetric spaces was completely settled in the authors thesis~\cite{Freyn09}. The study of Kac-Moody symmetric spaces unites functional analytic, algebraic and geometric aspects. This first part focuses on the functional analytic basics; the second part focuses on algebraic aspects (including the classification) while the third part investigates the geometry. This paper contains the core of the first part, describing the functional analytic constructions at the basis of Kac-Moody symmetric spaces.

Let us start by giving some general remarks about the theory:  Affine Kac-Moody symmetric spaces are the fundamental global objects of a wider theory, called \emph{(affine) Kac-Moody geometry}~\cite{Freyn10a}, describing various classes of mathematical objects, that are united by their shared structure properties. Let us start by giving a short overview of Kac-Moody geometry. Affine Kac-Moody geometry is the geometry of affine Kac-Moody, that are equipped with a manifold structure; to this end one uses completions of affine Kac-Moody algebras and Kac-Moody groups with respect to certain (sets of) seminorms. Affine Kac-Moody groups can be described as certain torus extensions $\widehat{L}(G, \sigma)$ of (possibly twisted) loop groups $L(G, \sigma)$. Correspondingly affine Kac-Moody algebras are $2$\ndash dimensional extensions of (possibly twisted) loop algebras $L(\mathfrak{g}, \sigma)$. Here $G$ denotes a compact or complex simple Lie group, $\mathfrak{g}$ its Lie algebra and $\sigma\in \textrm{Aut}(G)$ a diagram automorphism, defining the ``twist''. Depending on the regularity assumptions on the loops (e.g. of Sobolev class $H^k$), one gets families of completions of the minimal (=algebraic) affine Kac-Moody groups $\widehat{L_{alg}G}^{\sigma}$.  The minimal algebraic loop group just consists of Laurent polynomials. Following Jacques Tits, completions defined by imposing regularity conditions on the loops are called ``analytic completions'' in contrast to the more abstract formal completion. Various analytic completions and objects closely related to them play an important role in different branches of mathematics and physics, especially quantum field theory, integrable systems and differential geometry. In most cases their use is motivated by the requirement to use functional analytic methods or by the need to work with manifolds and Lie groups \cite{PressleySegal86, SegalWilson85, PalaisTerng88, Tsvelik95, Guest97, Popescu05, Kobayashi11, Khesin09, HPTT, HeintzeLiu, Heintze06} and references therein. For the study of completions of Kac-Moody algebras see also~\cite{GoodmanWallach84, Suto87, Suto88, Suto88b, Rodriguez89, Rodriguez89b, Suto92} 

To describe the scope of the theory and the content of this article let us start with some general remarks: From an algebraic point of view affine Kac-Moody algebras and simple Lie algebras are closely related as they are both realization of (generalized) Cartan matrices~\cite{Kac90},\cite{Carter05}. As a consequence, they share important parts of their structure theories. The philosophy of Kac-Moody geometry asserts that this similarity extends to all types of geometric objects whose structure reflects actions of those groups. More precisely Kac-Moody geometry claims the existence of infinite dimensional counterparts to all finite dimensional differential geometry objects, whose symmetries are described by real or complex simple Lie groups~\cite{Heintze06}, \cite{Freyn10a}. The symmetries of those infinite dimensional objects are supposed to be the corresponding complex or real affine Kac-Moody groups. Thus, let us start by recalling the finite dimensional blueprint.

 In this case the symmetry group is a simple complex or real Lie group. A complex simple Lie group has up to conjugation a unique compact real from. All further real forms are non-compact. 
 
The fundamental geometric objects associated to simple Lie groups are Riemannian symmetric spaces. Recall that a Riemannian manifold $(M,g)$ is a symmetric space if the isometry group contains for any point $p\in M$ contains an element $\sigma_p$ such that $s\sigma_p|_{T_pM}=i\Id$. The construction and classification of Riemannian symmetric spaces by \'Elie Cartan was a central result of Riemannian geometry in the first half of the 20'th century, opening up the way to new directions of research. Besides being a class of examples of Riemannian manifolds of fundamental importance, Riemannian symmetric spaces display numerous connections to Lie theory, mathematical physics, analysis and combinatorics. Hence there was intensive research (also begun by \'Elie Cartan) devoted to the construction of a class of infinite dimensional symmetric spaces, displaying similar structural properties. From the point of view of structural properties and the numerous (in part conjectured) connections to other fields such as supergravity and $M$-theory, especially symmetric spaces associated to Kac-Moody groups are interesting.

 Attached to any compact or complex simple Lie group, or more generally to any symmetric space, there is a spherical building~\cite{AbramenkoBrown08}. Combinatorically a building is described as a simplicial complex or as a graph, satisfying several axioms. Geometrically the building corresponding to a symmetric space $M$ has a clear intuitive description~\cite{BridsonHaeflinger99}: suppose the dimension $k$ of maximal flats (maximal flat subspaces), called the rank, satisfies $k \geq2$ (otherwise the building degenerates to  a set of points) and choose a point $p\in M$. The building can be seen in the tangent space $T_pM$ as follows: Consider first all $k$\ndash dimensional tangent spaces to maximal flats through the point $p$ (for a compact simple Lie group choose $p=Id$ the identity, then $T_pM$ can be identified with the Lie algebra, maximal flats through the identity are maximal tori and the tangent spaces to maximal flats are just Abelian subalgebras of the Lie algebra $\mathfrak{g}$) and remove all intersection points (the so\ndash called singular points). Then each flat is partitioned into chambers and the set of all those chambers for all flats through $p$ corresponds exactly to the chambers of the building. All other cells of the building can be seen in the set of singular points. Chambers belonging to one flat form an apartment. Nevertheless, in this description, one does not see all apartments. To this end one has to pass to the sphere at infinity, to disenthral the construction from the dependence of a base point. 

Via the geometric description of the building in the tangent space the connections between the building and polar actions become manifest:
Let $M$ be a symmetric space and  $p\in M$. The isotropy representation of $M$ is the representation of $K$,  the group of all isometries of $M$ fixing $p$ on $T_pM$: 

\begin{displaymath}
K:T_pM\longrightarrow T_pM,
\end{displaymath}

For compact Lie groups the isotropy representation coincides with the Adjoint representation.

The principal orbits of the isotropy representation
have the following properties:
\begin{enumerate}
 \item The principal orbits meet the tangent spaces to the flats orthogonally and have complementary dimension. In particular, the normal spaces to the orbits are integrable. Actions with this property are called polar. Conversely, any polar action  is orbit equivalent to the isotropy representation of a symmetric space~\cite{Berndt03}. 
\item The geometry of the principal orbits as submanifolds of Euclidean space is particularly nice and simple. They are so\ndash called isoparametric submanifolds~\cite{PalaisTerng88}. Conversely, by a theorem of Thorbergsson~\cite{Thorbergsson91}, isoparametric submanifolds are principal orbits of polar representations if they are full, irreducible and the codimension is not $2$. 
\end{enumerate}

To summarize the finite dimensional blueprint:

\begin{itemize}
 \item Polar representations correspond to symmetric spaces, isoparametric submanifolds correspond roughly to polar representations. 
 \item Chambers in buildings correspond to points in isoparametric submanifolds.
 \item Buildings describe the structure at infinity of symmetric spaces of noncompact type.
 \item The geometric realization of buildings can be equivariantly embedded via the isotropy representation into the tangent space of a symmetric space.
\end{itemize}

Let us now turn to the infinite dimensional counterpart:
Affine twin buildings have been investigated by Ronan and Tits as a tool to understand minimal affine Kac-Moody groups \cite{Ronan03}. They play a similar role for affine Kac-Moody groups as classical buildings play for finite dimensional semisimple Lie groups. 

Affine Kac-Moody groups are distinguished among other classes of infinite dimensional Lie groups by two important properties:

\begin{itemize}
 \item They share the most important structure properties of finite dimensional simple Lie groups, i.e.\ they have $BN$-pairs, Iwasawa and Bruhat decompositions and highest weight representations~\cite{PressleySegal86}, \cite{Kac90}, \cite{Kumar02}, \cite{Moody95}. The main differences can be traced back to the Weyl group being finite for Lie groups and infinite for Kac-Moody groups.
 \item They have good explicit linear realizations in terms of $2$\ndash dimensional extensions of loop algebras and loop groups; thus they allow the use of functional analytic methods and the definition of manifold structures. The resulting Kac-Moody algebras and Kac-Moody groups are called of ``analytic'' type, in contrast for example to formal completions~\cite{Tits84}. Depending on the type of completions, analytic Kac-Moody algebras (groups) are Hilbert\ndash, Banach\ndash, Fr\'echet\ndash, etc.\ Lie algebras (groups)~\cite{PressleySegal86}.
\end{itemize}
 
Work done during the last 25 years by various authors, notably Ernst Heintze and Chuu-Lian Terng shows that analytic completions of affine Kac-Moody groups describe the symmetries of various interesting objects, that have finite dimensional counterparts, thus confirming the philosophy of Kac-Moody geometry~\cite{PressleySegal86,  SegalWilson85, Terng89, HPTT, Terng91,  Terng95, HeintzeLiu, Popescu06,   Heintze06,  Khesin09, Freyn09, Freyn10a}. 

For example, it is well known that affine analytic Kac-Moody groups are symmetry groups of (Kac-Moody) symmetric spaces~\cite{Freyn09}. The structure theory and the classification of those spaces parallels closely the theory of finite dimensional Riemannian symmetric spaces, making them the appropriate generalization. Furthermore there is a theory of polar actions on Hilbert spaces~\cite{Terng95} and proper Fredholm isoparametric submanifolds in Hilbert spaces~\cite{Terng89}. Principal orbits of polar action on Hilbert spaces are proper Fredholm isoparametric submanifolds in Hilbert spaces.  Thus a broad generalization of the finite dimensional blueprint appears. We find the same classes of objects, satisfying similar relations among each other and sharing similar structure properties ~\cite{Heintze06} \cite{Freyn10a}.  Nevertheless, a serious gap in the infinite dimensional theory described so far is the lack of an object generalizing spherical buildings to Kac-Moody geometry. For algebraic Kac-Moody groups there are algebraic twin buildings, but they do not work well with completions as the action of any completed Kac-Moody group does not preserve the twinning; thus affine twin buildings are of restricted use in the realm of a Kac-Moody geometry. This problem can be seen also on the group level, as analytic Kac-Moody groups do not allow for the definition of a twin $BN$\ndash pair or a $BN$\ndash pair~\cite{AbramenkoBrown08}, \cite{PressleySegal86}. We gave a solution to this problem introducing the notion of \emph{twin cities} in~\cite{Freyn10d, Freyn10b}.

 As Kac-Moody symmetric spaces are tame Fr\'echet manifolds, we start in section~\ref{sect:tame_frechet_manifolds} by reviewing the theory of tame Fr\'echet structures following the presentation given by Richard Hamilton~\cite{Hamilton82}; then in section~\ref{sect:Some_tame_Frechet_spaces} we introduce certain spaces of holomorphic maps and prove, that they are tame Fr\'echet; these spaces reappear in our construction of the tame structures on affine Kac-Moody groups.  In section~\ref{chap:tame} we prove then, that holomorphic loop algebras are tame Fr\'echet Lie algebras and that holomorphic loop groups are tame Fr\'echet Lie groups. To achieve this result, we need to establish various technical results about Fr\'echet manifolds. In section~\ref{PolaractionsontameFrechetspaces} we extend the theory of polar actions from the classical Hilbert space setting~\cite{Terng95} to our setting of tame Fr\'echet spaces; in the final section~\ref{Kac-Moody groups} we turn our attention to holomorphic Kac-Moody groups and prove that they are tame Fr\'echet Lie groups.

\section{Tame Fr\'echet manifolds}
\label{sect:tame_frechet_manifolds}

\subsection{Fr\'echet spaces}

In this introductory section we collect some standard results about Fr\'echet spaces, Fr\'echet manifolds and Fr\'echet Lie groups. Further details or omitted proofs can be found in Hamilton's article~\cite{Hamilton82}.

\begin{definition}[Fr\'echet space]
A Fr\'echet vector space is a locally convex topological vector space which is complete, Hausdorff and metrizable.
\end{definition}

\begin{lemma}[Metrizable topology]
A topology on a vector space is metrizable iff it can be defined by a countable collection of seminorms.
\end{lemma}

\begin{proof}
For the proof see~\cite{Hamilton82}.
\end{proof}

\noindent Let us look at some examples:

\begin{example}[Fr\'echet spaces]~
\label{frechetexamples}
\begin{enumerate}
\item Every Banach space is a Fr\'echet space. The countable collection of norms contains just one element.
\item Let $\textrm{Hol}(\mathbb{C}, \mathbb{C})$ denote the space of holomorphic functions $f: \mathbb{C} \longrightarrow \mathbb{C}$. Let furthermore $K_n$ be a sequence of simply connected compact sets in $\mathbb C$, such that $K_n \subset K_{n+1}$ and $\bigcup K_n=\mathbb{C}$. Let $\|f\|_n:= \displaystyle\sup_{z\in K_n} |f(z)|$. Then $\textrm{Hol}(\mathbb {C}, \mathbb {C}; \|\hspace{3pt}\|_n)$ is a Fr\'echet space.
\item More generally, for every Riemann surface $S$ the sheaf of holomorphic functions carries a Fr\'echet structure which is defined similarly as in the special case $S=\mathbb{C}$ discussed in the second example.
\end{enumerate}
\end{example}
 
\begin{definition}
 A Fr\'echet manifold is a (possibly infinite dimensional) manifold with charts in a Fr\'echet space such that the chart transition functions are smooth. 
\end{definition}

While it is possible to define Fr\'echet manifolds in this way, there are two strong impediments to the development of analysis and geometry of those spaces:
\begin{enumerate}
\item In general there is no inverse function theorem for smooth maps between Fr\'echet spaces. For counterexamples and examples showing features special to Fr\'echet spaces, see~\cite{Hamilton82}.
\item In general the dual space of a Fr\'echet space is not a Fr\'echet space. 
\end{enumerate}

Dealing with those problems  is difficult:  
let us investigate as a typical example the space of holomorphic functions $\textrm{Hol}(\mathbb{C}, \mathbb{C})$.  $\mathbb{C}$ can be interpreted as a direct limit of the sets $K_n:=B_n(0)$ with respect to inclusion; the space of functions on a direct limit is an inverse limit; $\textrm{Hol}(\mathbb{C}, \mathbb{C})$ should thus be interpreted as the inverse limit of a sequence of function spaces $\textrm{Hol}(K_n, \mathbb{C})$, where 
\begin{displaymath}\textrm{Hol}(K_n, \mathbb{C}):=\bigcup_{K_n\subset U_{n}^m}\{\textrm{Hol}(U_{n}^m, \mathbb{C})\}\, .\end{displaymath}

 By a choice of appropriate norms on the spaces $\textrm{Hol}(K_n, \mathbb{C})$, one can give them structures as Bergmann-(or Hardy-) spaces. See for example~\cite{HKZ00} and \cite{Duren00}. Hence $\textrm{Hol}(\mathbb{C}, \mathbb{C})$ can be interpreted as inverse limit of Hilbert spaces. By category theory, the categorical dual space of an inverse limit is a direct limit and vice versa. Thus it is clear that we cannot expect the dual space of a Fr\'echet space to be Fr\'echet. It is Fr\'echet iff the projective limit and the inductive limit coincide, that is, if they both stabilize. This is the case exactly for Banach spaces~\cite{Schaefer80}.

The solution to the first problem is based on a more refined control of the projective limits. Using this structure, the inverse function theorems on the Hilbert-(resp.\ Banach-) spaces in the sequence piece together to give an inverse function theorem on the limit space; this is the famous Nash-Moser inverse function theorem. In the next sections we formalize those concepts. 

Let us note that there are other ways to deal with these analytic problems. A recently proposed method is the concept of bounded Fr\'echet geometry developed by Olaf M\"uller~\cite{Muller06, Muller08}. 

The solution to the second problem consists in avoiding dual spaces. Let us remark that the theory could also be formulated by defining dual spaces as direct limit locally convex topological vector spaces.

\subsection{Tame Fr\'echet spaces}

The central problem for all further structure theory of Fr\'echet spaces is a better control of the set of seminorms. For a single Fr\'echet space this is done by the concept of a grading; much deeper is the question of comparing collections of seminorms on two different Fr\'echet spaces: for two Fr\'echet spaces $F$ and $G$ and a map $\varphi: F\longrightarrow G$ this is done by imposing estimates similar in spirit to the concept of quasi isometries relating the sequences of norms $\|\varphi(f)\|_n$ and $\|f\|_m$, leading to the concept of ``tame'' maps. The property of a Fr\'echet space to be tame is defined then by ``tame equivalence'' to some model space $\Sigma(B)$, which is essentially a space of holomorphic functions; the standard reference for these considerations is the article~\cite{Hamilton82}. We are slightly sharpening the concept of ``tame equivalence'' introduced by Hamilton to define $(r,b,C(n))$-equivalence and emphasize the nature of tame spaces as a generalization of spaces of holomorphic functions.

The prerequisite for estimating norms under maps between Fr\'echet spaces is a comparison of the norms on the Fr\'echet space itself. This leads to the notion of a grading:

\begin{definition}[grading]
Let $F$ be a Fr\'echet space. A grading on $F$ is a collection of seminorms $\{\|\hspace{3pt}\|_{n}, n\in \mathbb N_0\}$ that define the topology and satisfy
$$\|f \|_0\leq \|f\|_1 \leq \|f\|_2 \leq\| f \|_3 \leq \dots \,.$$
\end{definition}

\begin{lemma}[Constructions of graded Fr\'echet spaces]~
\begin{enumerate}
	\item A closed subspace of a graded Fr\'echet space is a graded Fr\'echet space.
	\item Direct sums of graded Fr\'echet spaces are graded Fr\'echet spaces.
\end{enumerate}
\end{lemma}

\begin{proof}
For the proof see~\cite{Hamilton82}.
\end{proof}

\noindent Every Fr\'echet space admits a grading. Let $(F, \|\hspace{3pt}\|_{n, n\in \mathbb{N}})$ be a Fr\'echet space. Then 
\begin{displaymath}
\left(\widetilde{F}, \widetilde{\|\hspace{3pt}\|}_{n, n\in \mathbb{N}}\right)
\end{displaymath} 
such that $\widetilde{F}=F$ as a set and $\widetilde{\|\hspace{3pt} \|}_{n}:=\displaystyle\sum_{i=1}^n \|\hspace{3pt}\|_i$ is a graded Fr\'echet space. The Fr\'echet topologies of $F$ and $\widetilde{F}$ coincide. The existence of a grading is thus not a property of the Fr\'echet space as a metrizable topological space. It is an \emph{additional structure} which is geometric in nature.

\begin{definition}[Tame equivalence of gradings]
Let $F$ be a graded Fr\'echet space, $r,b \in \mathbb{N}$ and $C(n), n\in \mathbb{N}$ a sequence with values in $\mathbb{R}^+$. The two gradings  $\{\|\hspace{3pt}\|_n\}$ and $\{\widetilde{\|\hspace{3pt}\|}\}$ are called $\left(r,b,C(n)\right)$-equivalent iff 
\begin{equation*}
 \|f\|_n \leq C(n) \widetilde{\|f\|}_{n+r} \text{ and }  \widetilde{\|f\|}_n \leq C(n)\|f\|_{n+r} \text{ for all } n\geq b
\,.
\end{equation*}
They are called tame equivalent iff they are $(r,b,C(n))$-equivalent for some $(r,b,C(n))$.
\end{definition}

\noindent The following example is basic:

\begin{example}
Let $B$ be a Banach space with norm $\| \hspace{3pt} \|_B$. Denote by $\Sigma(B)$ the space of all exponentially decreasing sequences $\{f_k\}$, ${k\in \mathbb N_0}$ of elements of $B$.
On this space, we can define different gradings:

\begin{align}
\|f\|_{l_1^n} &:= \sum_{k=0}^{\infty}e^{nk} \|f_k\|_B\\
\|f\|_{l_{\infty}^n}&:= \sup_{k\in \mathbb N_0} e^{nk}\|f_k\|_B
\end{align}
\end{example}

\begin{lemma}
On the space $\Sigma(B)$ the two gradings $\|f\|_{l_1^n}$ and $\|f\|_{l_{\infty}^n}$ are tame equivalent.
\end{lemma}

\begin{proof}
For the proof see~\cite{Hamilton82}.
\end{proof}

\begin{example}
This example is used in section~\ref{sect:Some_tame_Frechet_spaces}.

\begin{itemize}
\item The space $\Sigma(\mathbb{C}^2, \textrm{Eucl})$ of exponentially decreasing sequences of elements in $\mathbb C^2$ where $\mathbb{C}^2$ is equipped with the Euclidean norm is a tame Fr\'echet space.
\item The space $\Sigma(\mathbb{C}^2, \textrm{Sup})$  of exponentially decreasing sequences of elements in $\mathbb C^2$ where $\mathbb{C}^2$ is equipped with the supremum-norm $\|(c_1, c_1')\|_B :=\sup(|c_1|,|c_1'|)$ is a tame Fr\'echet space.
\end{itemize}
\end{example}

\noindent Let us give an intuitive characterization of tame spaces

\begin{lemma}
Let $B$ be a complex Banach space. The space $\Sigma(B)$ of exponentially decreasing sequences in $B$ is isomorphic to the space $\textrm{Hol}(\mathbb{C},B)$ of $B$\ndash valued holomorphic functions.
\end{lemma}

\begin{proof}
We have to prove both inclusions:
Let first $f\in \textrm{Hol}(\mathbb{C},B)$ with its Taylor series expansion:
\begin{displaymath}
f(z)=\sum_{k=0}^{\infty}f_k z^k.
\end{displaymath} 
The coefficients $f_k$ are elements of $B$ and we have to show that $(f_k)$ is an exponentially decreasing sequence. We want to use the $l_{\infty}$\ndash norm.
For the estimate of $f_k$ we need three ingredients:
\begin{enumerate}
 \item As $f(z)$ is entire the expansion $\|f(z)\|=\sum_{n=0}^{\infty}f_k z^k$ converges for all $z\in \mathbb{C}$ - hence $\sup_{z\in B(0,e^n)}|f(z)|<\infty$ of all $n$.
\item Differentiation of $f$ yields the identity: $f_k=\frac{1}{k!}f^{(k)}(0)$.
\item The Cauchy inequality~\cite{Berenstein91}, 2.1.20: Let $f$ be holomorphic on $B(0,r)$. Then
\begin{displaymath}
 |f^{(k)}(0)|\leq k!\frac{\sup_{z\in B(0,r)}|f(z)|}{r^k}\,.
\end{displaymath}
\end{enumerate}
Combining these ingredients we get for $n\in \mathbb{N}$
\begin{align*}
 \sup_{k}|f_k|e^{nk}&=\sup\left|\frac{f^{(k)}(0)}{k!} \right|e^{kn}\leq\\
&\leq\sup_k\frac{e^{kn}}{k!}\left|k!\frac{\sup_{z\in B(0,e^n)}|f(z)|}{e^{nk}}\right|=\\
&=\sup_{z\in B(0,e^n)}|f(z)|\leq \infty\,.
\end{align*}
Hence 
\begin{displaymath}
 \|f\|_{l_{\infty}^{n}}<\infty\, .
\end{displaymath}

As $f(z)$ is entire the expansion $\|f(z)\|=\sum_{n=0}^{\infty}f_k z^k$ converges for all $z\in \mathbb{C}$ yielding the result for all $n$.\\

Conversely for any exponentially decreasing sequence $(f_n)_{n\in \mathbb{N}}$, $f_n\in B$ we define the function
\begin{displaymath}
f(z):=\sum_{n=0}^{\infty}f_n z^n\, .
\end{displaymath} 
We have to prove that this defines an entire holomorphic function as we get for any $z\in B(0,e^n)$ the estimate
\begin{displaymath}
 |f(z)|=\left|\sum_k f_k z^k\right|\leq \sum_k |f_k||z^k|\leq \sum |f_k|e^{kn}<\infty\,.
\end{displaymath}
\end{proof}

\begin{corollary}
Let $B$ be a real Banach space. Then $\Sigma(B)$ consists of all $B\otimes \mathbb{C}$\ndash valued holomorphic functions on $\mathbb{C}$ such that 
\begin{displaymath}
 \overline{f(z)}=f(\overline{z})\, .
\end{displaymath}
  
\end{corollary}

\begin{example}
 The easiest example is the choice $B=\mathbb{C}$. Then we find
\begin{displaymath}
 \Sigma(B):=\left\{(a_k)_{k\in \mathbb{N}_0}|\sum_k |a_k|e^{kn}<\infty \forall n\in \mathbb{N}\right\}\,.
\end{displaymath}
For fixed $n$ the sequence $(a_k)$ corresponds to the coefficients of the Laurent series of a holomorphic function converging on the disc
$B(0,e^n)$. For $f(z):=\sum a_k z^k$ this is proved by the estimates
\begin{displaymath}
 |f(z)|=|\sum_k a_k z^k|\leq \sum_k |a_k||z^k|\underbrace{\leq}_{\textrm{for}\quad z\in B(0,e^n)} \sum |a_k|e^{kn}<\infty
\end{displaymath}
As this condition is satisfied by assumption for all $n\in \mathbb{N}$ we get that the function defined by the series $(a_k)_{k\in \mathbb{N}}$ is holomorphic on $\mathbb{C}$. Hence $\Sigma(B)=\textrm{Hol}(\mathbb{C},\mathbb{C})$.
\end{example}

\noindent Let $F$, $G$, $G_1$ and $G_2$  denote graded Fr\'echet spaces.

\begin{definition}[Tame linear map]
A linear map $\varphi: F\longrightarrow G$ is called $(r,b,C(n))$-tame if it satisfies the inequality
$$\|\varphi(f)\|_n \leq C(n)\|f\|_{n+r}\,.$$
$\varphi$ is called tame iff it is $(r,b,C(n))$-tame for some $(r,b, C(n))$.
\end{definition}

Let us give an example showing that tameness of maps depends on the grading. Let us look at the following variant of an example of~\cite{Hamilton82}:

\begin{example}
Let $\mathcal{P}$ be the space of entire holomorphic functions periodic with period
$2\pi i$ and bounded in each left half-plane. Define $L: \mathcal{P} \longrightarrow \mathcal{P}$ by $Lf(z) = f(2z)$.
\begin{enumerate}
\item Define first the grading on $\mathcal{P}$ by:
\begin{displaymath}
\|f \|_n = \sup{|f(z)|:\Re(z) = n}\,.
\end{displaymath}
 Then $\| Lf\|_n\leq \| f \|_{2n}$, hence $L$ is not tame.
 \item Define now the grading on $\mathcal{P}$ by
\begin{displaymath}
\|f \|_n = \sup{|f(z)|:\Re(z) = 2^n}\,.
\end{displaymath}
Then $\| Lf\|_n\leq \| f \|_{n+1}$, hence $L$ is $(1,0,1)$-tame.
 \end{enumerate}
Clearly the two gradings are not tame equivalent. 
\end{example}

\begin{definition}[Tame isomorphism]
A map $\varphi:F\longrightarrow G$ is called a tame isomorphism iff it is a linear isomorphism and $\varphi$ and $\varphi^{-1}$ are tame maps.
\end{definition}

\begin{definition}[Tame direct summand]
$F$ is a tame direct summand of $G$ iff there exist tame linear maps $\varphi: F\longrightarrow G$ and $\psi: G \longrightarrow F$ such that $\psi \circ \varphi: F \longrightarrow F$ is the identity.
\end{definition}

\begin{definition}[Tame Fr\'echet space]
\label{tame Frechet space}
$F$ is a tame Fr\'echet space iff there is a Banach space $B$ such that $F$ is a tame direct summand of $\Sigma(B)$.
\end{definition}

\begin{lemma}[Constructions of tame Fr\'echet spaces]~
\label{constructionoftamespaces}
\begin{enumerate}
	\item A tame direct summand of a tame Fr\'echet space is tame.
	\item A cartesian product of two tame Fr\'echet spaces is tame.
\end{enumerate}
\end{lemma}

\noindent There are many different examples of tame Fr\'echet spaces. One can show that all examples described in \ref{frechetexamples} are tame Fr\'echet spaces. For proofs and additional examples see~\cite{Hamilton82}. In section \ref{sect:Some_tame_Frechet_spaces} we study in detail the tame Fr\'echet spaces of holomorphic functions which we need for the construction of Kac-Moody symmetric spaces.

Let us now introduce the notion of a tame Fr\'echet Lie algebra:

\begin{definition}[Tame Fr\'echet Lie algebra]
\label{tamefrechetLie algebra}
A Fr\'echet Lie algebra $\mathfrak{g}$ is a tame Lie algebra iff it is a tame vector space and
$\textrm{ad}(X)$ is a tame linear map for every $X\in \mathfrak{g}$.
\end{definition}

The condition on $\textrm{ad}(X)$ assures the tame structure to be invariant under the adjoint action. We clearly have the following class of examples:

\begin{example}
 Any finite dimensional Lie algebra is tame.
\end{example}

\begin{example}
The realizations of the loop algebras $L(\mathfrak{g},\sigma)$ are tame Fr\'echet Lie algebras for $H^0$\ndash Sobolev  loops, smooth loops and holomorphic loops --- compare section \ref{holomorphicstructuresonkacmoodyalgebras}.
\end{example}

Up to now all maps, we studied were linear maps. Let us now proceed by a short review of some nonlinear tame Fr\'echet objects:

\begin{definition}
A nonlinear map $\Phi: U\subset F \longrightarrow G$ is called $(r, b, C(n))$-tame iff it satisfies the inequality 
\begin{displaymath}
\|\Phi(f)\|_{n}\leq C(n)(1+\| f\|_{n+r})\,\forall n>b\,.
\end{displaymath}
$\Phi$ is called tame iff it is $(r,b,C(n))$-tame for some $(r,b, C(n))$.
\end{definition}

\begin{example}
 Suppose $F$ and $G$ are Banach space (hence the collection of norms consists of one norm) and $\Phi:F\longrightarrow G$ is a $(r_1,b_1,C_1)$\ndash tame isomorphism with a $(r_2, b_2,C_2)$\ndash tame inverse $\Phi^{-1}$. If $b_1\geq 2$ and $b_2\geq 2$ the condition on the norms vanishes. If $b_1=r_1=b_2=r_2=0$ we get 
\begin{displaymath}
\|\Phi(f)\|\leq C_1(1+\| f\|)\,\quad\textrm{and}\quad \|\Phi^{-1}(g)\|\leq C_2(1+\| g\|)\, .
\end{displaymath}
After some manipulations we find:
\begin{displaymath}
 \frac{1}{C_2}\|f\|-1\leq \|\Phi(f)\|\leq C_1(1+\|f\|)
\end{displaymath}
which is similar to the definition of a quasi isometry~\cite{Burago01}
\end{example}

\begin{lemma}[Construction of tame maps]~
\label{constructionoftamemaps}
\begin{enumerate}
\item Let $\Phi: U\subset F \longrightarrow G_1\times G_2$ be a tame map. Define the projections $\pi_i:G_1\times G_2 \longrightarrow G_i, i=1,2$. The maps
\begin{displaymath}
\Phi_i:=\pi_i \circ \Phi: U\longrightarrow G_i
\end{displaymath}
are tame as well.
\item Let $\Phi_i: U \subset F\longrightarrow G_i, i\in \{1,2\}$ be $(r_i, b_i, C_i(n))$-tame maps. Then the map
\begin{displaymath}
\Phi:=(\Phi_1, \Phi_2): U \longrightarrow G_1 \times G_2
\end{displaymath}
is $(\max(r_1, r_2),\max(b_1, b_2), C_1(n)+C_2(n))$-tame.
\end{enumerate}
\end{lemma}

\begin{proof}~
\begin{enumerate}
\item Projections onto a direct factor are $\left(0,0,(1)_{n\in \mathbb{N}}\right)$-tame. The composition of tame maps is tame. Thus $\Phi_{i}$ is tame.
\item Let $f\in U\subset F$. \begin{align*}
\|\Phi(f)\|_n &= \|\Phi(f)\|^1_n+\|\Phi(f)\|^2_n \leq\\
              &\leq C_1(n)(1+\| f\|_{n+r_1})+ C_2(n)(1+\|f\|_{n+r_2}) \leq\\
							&\leq C_1(n)(1+\| f\|_{n+\max(r_1, r_2)})+ C_2(n)(1+\|f\|_{n+\max(r_1, r_2)})= \\
							&\leq (C_1(n)+C_2(n))(1+\| f\|_{n+\max(r_1, r_2)})
\end{align*}
for all $n\geq \max(b_1, b_2)$.
\end{enumerate}
\end{proof}

\subsection{Tame Fr\'echet manifolds}

\begin{definition}[Tame Fr\'echet manifold]
A tame Fr\'echet  manifold is a Fr\'echet manifold with charts in a tame Fr\'echet space
such that the chart transition functions are smooth tame maps.
\end{definition}

\begin{example}
Every Banach manifold is a tame Fr\'echet manifold.
\end{example}

\begin{definition}
Let $M$ and $N$ be two tame Fr\'echet manifolds modeled on $F$ resp. $G$. A map $f: M\longrightarrow N$ is tame iff for every pair of charts $\psi_i:V_i \subset N \longrightarrow V_i'$ and $\varphi_j: U_i \subset M \longrightarrow U_i'$, the map $\psi_i \circ f \circ \varphi_j^{-1}$ is tame whenever it is defined. 
\end{definition}

For the construction of tame structures on Kac-Moody groups we need to introduce the following new notion:

\begin{definition}[Tame Fr\'echet submanifold of finite type]
\label{Tamesubmanifoldoffinitetype}
Let $n\in \mathbb{N}$. A subset $M\subset F$ is a $n$-codimensional smooth submanifold of $F$ iff for every $m\in M$ there are open sets $U(m)\subset F$, $V(m)\subset G\times \mathbb{R}^n$ and a tame Fr\'echet chart $\varphi_m:U(m)\longrightarrow V(m)\subset G\times \mathbb{R}^n$ such that
$$\varphi_{M}(M\cap U(m))= G\cap V(m)\,.$$
\end{definition}

\begin{lemma}
A tame submanifold of finite type is a tame Fr\'echet manifold.
\end{lemma}

This result was proved in~\cite{Freyn12c}.

\begin{lemma}
\label{mapinfrechetsubmanifold}
Let $M\subset F$ be a tame Fr\'echet submanifold of finite type. Let $H$ be a tame Fr\'echet space. A map $\varphi_M:H \longrightarrow M$ is tame if it is tame as a map $\varphi_F:H\rightarrow F$, where $\varphi_F$ is defined via the embedding $M\subset F$. 
\end{lemma}

This means, that we chose the tame structure on a tame Fr\'echet submanifold of finite type compatible with the tame structure on the vector the manifold is embedded in. We need this result in the proof that affine Kac-Moody groups carry a tame structure. 

\begin{proof}
Let $\varphi$ be tame as a map $\varphi_F:H\rightarrow F$. Then the cocatenation $\varphi_i\circ \varphi$ is tame for any chart $\varphi$ of $M$. Thus $\varphi$ is tame as a map into $M$.   
\end{proof}

\noindent The technical reason for working in the category of tame Fr\'echet spaces and tame maps is the following Nash-Moser inverse function theorem. 
We cite the version of~\cite{Hamilton82}.

\begin{theorem}[Nash-Moser inverse function theorem]
Let $F$ and $G$ be tame Fr\'echet spaces and $\Phi: U\subseteq F \longrightarrow G$ a smooth tame map. Suppose that the equation for the derivative $D\Phi(f)h=k$ has a unique solution $h=V\Phi(f)k$ for all $f\in U$ and all $k$ and that the family of inverses $V\Phi:U\times G \longrightarrow F$ is a smooth tame map. Then $\Phi$ is locally invertible, and each local inverse $\Phi^{-1}$ is a smooth tame map.
\end{theorem}

\noindent A description of this theorem a proof and some of its applications is the subject of the article~\cite{Hamilton82}.
In comparison to the classical Banach inverse function theorem the important additional assumption is that the invertibility of the differential is assumed not only in a single point $p$ but in a small neighborhood $U$ around $p$. This additional condition is necessary~\cite{Hamilton82} because
 in contrast to the Banach space situation it is not true in the Fr\'echet space case that the existence of an invertible differential in one point leads to invertibility in a neighborhood. Let us note the following result~\cite{Hamilton82}, theorem 3.1.1.\ characterizing the family of smooth tame inverses.

\begin{theorem}
 Let $L: (U \subseteq F) \times H \longrightarrow K$ be a smooth tame family of linear
maps. Suppose that the equation 
\begin{displaymath}
L(f)h = k
\end{displaymath} has a unique solution $h$ for all $f$ and
$k$ and that the family of inverses $V(f)k = h$ is continuous and tame as a map from $K$ to$H$
Then $V$ is also a smooth tame map 
\begin{displaymath}
V:(U\subseteq F)\times K\longrightarrow H.
\end{displaymath}
\end{theorem}

We discuss these construction in more detail in ~\cite{Freyn12g}.

\section{Tame Fr\'echet spaces for Kac-Moody symmetric spaces}
\label{sect:Some_tame_Frechet_spaces}

In this section we prove the tameness of certain Fr\'echet spaces of holomorphic functions on $\mathbb{C}^*$. We show first, that the space $Hol(\mathbb{C}^*, \mathbb{C})$ is a tame Fr\'echet space; then we extend this result to the space $Hol(\mathbb{C}^*, V_{\mathbb{C}})$, where $V_{\mathbb{C}}$ denotes a complex vector space. We use these results in section~\ref{Lie_algebras_of_holomorphic_maps} for the construction of tame Fr\'echet affine Kac-Moody algebras and their Kac-Moody groups via the loop algebra realization (resp. loop group realization).

In~\cite{Hamilton82} the following result is shown:

\begin{lemma}
\label{holccistame}
$F:= \textrm{Hol}(\mathbb C, \mathbb C)$ is a tame Fr\'echet space.
\end{lemma}

\begin{proof}
~\cite{Hamilton82}.
\end{proof}

\begin{corollary}
The space $F: =\textrm{Hol}(\mathbb C,\mathbb C^n)$ is a tame Fr\'echet space.
\end{corollary}

The proof of R. Hamilton makes strong use of the following observation:

\begin{lemma}
\label{tame_equivalence_of_gardings_on_holcc}
Let $F:= \textrm{Hol}(\mathbb C, \mathbb C)$. The two gradings
\begin{align*}
\|f\|_{L_1^n}:=& \frac{1}{2\pi} \int_{\partial B_n}|f(z)|dz\\
\intertext{and}
\|f\|_{L_{\infty}^n}:=& \sup_{z \in B_n }\|f(z)\|
\end{align*}
are tame equivalent.
\end{lemma}

\begin{proof}
~\cite{Hamilton82}.
\end{proof}

Our aim is now to prove that $\textrm{Hol}(\mathbb{C}^*, \mathbb{C})$ is a tame Fr\'echet space. Our strategy parallels the one used by R. Hamilton for the proof of lemma~\ref{holccistame} and \ref{tame_equivalence_of_gardings_on_holcc}.

Let us introduce some notation:

\begin{notation}~
\label{A_n}
\begin{itemize}
\item Let $A_n$ denote the annulus $A_n:= \{z\in \mathbb{C}^*|e^{-n}\leq |z| \leq e^n\}$ and define the boundaries of $A_n$ by 
\begin{displaymath}
\partial A_n^+:= \{z|\hspace{3pt}|z|=e^n\}\quad \textrm{and} \quad \partial A_n^-:= \{z|\hspace{3pt}|z|=e^{-n}\},
\end{displaymath}
\item Let $A_n'$ denote the set $A_n':= \{z\in \mathbb C | -n \leq \Re(z)\leq n, 0\leq \Im(z)\leq 2\pi i\}$,
\item Let $B_n$ denote the disc $B_n:=\{z\in \mathbb C| \hspace{3pt} |z|\leq  e^n\}$.
\end{itemize}
\end{notation}

At some points we apply the Cauchy integral formula and the maximum principle for holomorphic functions.

We start with our version of lemma~\ref{tame_equivalence_of_gardings_on_holcc}. The proof of our result is considerably more involved because holomorphic functions on $\mathbb{C}^*$  may be unbounded in the limit $|z|\rightarrow \infty$ and $|z|\rightarrow 0$.

\begin{lemma}
Let $F:=\textrm{Hol} (\mathbb C^*, \mathbb C)$. The two gradings
\begin{align*}
\|f\|_{L_{\infty}^n}&:= \sup_{z\in A_n} |f(z)|\\
\intertext{and}
\|f\|_{L_1^n}&:= \frac{1}{2\pi} \sup \left\{\int_{\partial A_n^+}|f(z)|dz,\int_{\partial A_n^-}|f(z)|dz   \right\}
\end{align*}
are tame equivalent.
\end{lemma}

%\needs{6\baselineskip} 
\begin{proof}~
\begin{enumerate}
	\item We show: $\|f\|_{L_{1}^n}\leq \|f\|_{L_{\infty}^n}$.
\begin{alignat*}{1}
\|f\|_{L_{1}^n}&=\frac{1}{2\pi} \sup \left\{\int_{\partial A_n^+}|f(z)|dz,\int_{\partial A_n^-}|f(z)|dz   \right\}\leq\\
&\leq \frac{1}{2\pi} \sup \left\{\int_{\partial A_n^+}\sup_{\zeta\in A_n^+}|f(\zeta)|dz,\int_{\partial A_n^-}\sup_{\zeta \in A_n^-}|f(\zeta)|dz   \right\} \leq\\
&\leq \sup \left\{ \sup_{z \in \partial A_n^+} |f(z)|, \sup_{z\in  \partial A_n^-}|f(z)| \right\} \leq\\
&\leq \sup_{z\in A_n} |f(z)|=\|f\|_{L_{\infty}^n}
\end{alignat*}

	\item We show: $\|f\|_{L_{\infty}^n}\leq \frac{2}{r} \|f\|_{L_{1}^n}$.

	To this end, we identify the space $\textrm{Hol}(\mathbb C^*, \mathbb C)$ with the space $\textrm{Hol}_{2\pi i}(\mathbb C, \mathbb C)$ of $2\pi i$ periodic functions. Under this identification $A_n$ is identified with $A_n'$.
\end{enumerate}
\begin{alignat*}{1}
\|f\|_{L_{\infty}^n}&= \sup_{z \in A_n'} |f(z)|=\\
&= \sup_{z \in A_n'}\left|\frac{1}{2\pi i} \left\{ \int_{n+r}^{n+r+2\pi i} \frac{f(\zeta)}{z-\zeta}d\zeta - \int_{-n-r}^{-n-r+2\pi i} \frac{f(\zeta)}{z-\zeta}d\zeta \right\}\right|\leq\\
&\leq \sup_{z \in A_n'}\frac{1}{2\pi} \left\{ \int_{n+r}^{n+r+2\pi i} \left|\frac{f(\zeta)}{r}\right|d\zeta + \int_{-n-r}^{-n-r+2\pi i} \left|\frac{f(\zeta)}{r}\right|d\zeta\right\}=\\
&=\sup_{z\in A_n'}\frac{1}{2\pi r}\left\{ \int_{n+r}^{n+r+2\pi i} \left|f(\zeta)\right|d\zeta + \int_{-n-r}^{-n-r+2\pi i} \left|f(\zeta)\right|d\zeta\right\}\leq\\
&\leq \frac{2}{r}\frac{1}{2\pi} \sup\left\{\int_{n+r}^{n+r+2\pi i}|f(\zeta)|, \int_{-n-r}^{-n-r+2\pi i}|f(\zeta)| \right\}=\\
&=\frac{2}{r}\|f\|_{L_{\infty}^{n+r}}
\end{alignat*}
\end{proof}

We now prove, that $\textrm{Hol}(\mathbb C^*, \mathbb C)$ is a tame Fr\'echet space.

\begin{lemma}
\label{holc*cisfrechet}
$F:= \textrm{Hol}(\mathbb C^*, \mathbb C)$ is a tame Fr\'echet space.
\end{lemma}

Recall from definition~\ref{tame Frechet space}, that we have to find a Banach space $B$, such that $\textrm{Hol}(\mathbb C^*, \mathbb C)$ is a tame direct summand of $\Sigma(B)$. We choose $B=\mathbb{C}^2$. This

\begin{proof}
Let $f:= \sum_{k\in \mathbb Z} c_k z^k$ and set $f_0^+:=\sum_{k\in \mathbb N_0} c_k z^k$  and $f^-:=\sum_{-k\in \mathbb N} c_k z^k$. Clearly $f_0^+(z)$ and $f^-(\frac{1}{z})$ are holomorphic functions on $\mathbb C$.
Let
$$
 \begin{matrix}
\varphi: &\textrm{Hol}(\mathbb C^*, \mathbb C) &\longrightarrow &\Sigma(\mathbb C^2)\\
       &(f) &\mapsto &\left(\{c_k\}_{k\geq 0}, \{c_k\}_{k < 0} \right) \,.
\end{matrix}
$$

We use the notation $\widetilde{c}_k:=(c_k, c_{-k})\subset \mathbb C^2$ and use the supremum-norm on $\mathbb C^2$. 

\begin{enumerate}
\item We show: $\|f\|_{L_{\infty}^n}\leq \|\{\widetilde {c}_k\}\|$.
\begin{alignat*}{1}
\|f\|_{L_{\infty}^n}&= \sup_{z \in A_n}|f(z)|\leq\\
&\leq \sup_{z \in A_n}\left\{|f_0^+(z)|+ |f^-(z)| \right\}\leq\\
&\leq 2  \sup_{z\in A_n}\left\{\sup \left\{|f_0^+(z)|, |f^-(z)|  \right\} \right\}=\\
&= 2 \sup\left\{ \sup_{z \in A_n}|f_0^+(z)|, \sup_{z \in A_n}|f^-(\frac{1}{z})|  \right\}=\\
&= 2  \sup\left\{\|f_0^+\|_{L_{\infty}^n}, \|f^-\|_{L_{\infty}^n} \right\}\leq\\
&\leq 2   \sup \left\{ \|\{c_k\}\|_{L_1^n}, \|\{c_{-k}\|_{L_1^n}\}\right\}\leq\\
&\leq 2  \|\{\widetilde{c}_k\}\|_{L_{\infty}^n} 
\end{alignat*}

\item We show: $\|\{c_k\}\|_{L_{\infty}^n}\leq \|f\|_{L_1^n}$.
\begin{alignat*}{1}
\|\{\widetilde{c}_k\}\|_{L_{\infty}^n} &= \sup_k e^{nk}|\widetilde{c}_k|=\\
&=\sup_k e^{nk}\left|\sup\left\{c_k, c_{-k} \right\} \right|=\\
&\leq \sup \left\{\sup_k e^{nk}|c_k|, \sup_k e^{nk}|c_{-k}| \right\}\leq\\
&\leq \sup \left\{\sup_k e^{nk}\frac{1}{2\pi}\left|\int_{n}^{n+2\pi i} e^{-kz} f(z) dz \right|, \sup_k e^{nk} \frac{1}{2\pi} \left| \int_{-n}^{-n+2\pi i} e^{kz} f(z)\right|| \right\} \leq\\
&
\begin{aligned}
\leq \sup &\left\{\sup_k \frac{1}{2}\int_{n}^{n+2\pi i} | e^{nk} e^{-kn}|| e^{-k i\Im(z)}| |f(z)|dz\right.,\\
& \hphantom{\{}\left. \sup_k \frac{1}{2\pi} \int_{-n}^{-n+2\pi i} |e^{nk} e^{-kn}||e^{ki\Im(z)} ||f(z)|dz \right\}\leq
\end{aligned} \\
&\leq \sup \left\{\sup_k \frac{1}{2}\int_{n}^{n+2\pi i} |f(z)|dz, \sup_k \frac{1}{2\pi} \int_{-n}^{-n+2\pi i} |f(z)|dz \right\} \leq\\
&\leq \sup \left\{\sup_k \frac{1}{2}\int_{n}^{n+2\pi i} \sup_{\zeta \in \partial {A'_n}^+}|f(\zeta )|dz, \sup_k \frac{1}{2\pi} \int_{-n}^{-n+2\pi i} \sup_{\zeta \in \partial{A'_n}^+}|f(\zeta)|dz 
\right\} =\\
&= \sup\left\{ \sup_{\zeta \in \partial {A'_n}^+} |f(z)|,  \sup_{\zeta \in \partial {A'_n}^-} |f(z)|  \right\}\leq\\
&\leq \sup_{\zeta \in A'_n |f(\zeta)|}= \|f\|_{L_{\infty}^n}
\end{alignat*}
\end{enumerate}
\end{proof}

We use the following result in the proof that an affine Kac-Moody algebra is a tame Fr\'echet Lie algebra (cf. definition~\ref{tamefrechetLie algebra}):

\begin{lemma}
\label{differentialistame}
The differential 
\begin{displaymath}
\frac{d}{dz}:\textrm{Hol}(\mathbb{C}^*,\mathbb{C}) \longrightarrow \textrm{Hol}(\mathbb{C}^*,\mathbb{C}), \quad f \mapsto f'
\end{displaymath}
is a tame linear map.
\end{lemma}

\noindent For the proof of lemma~\ref{differentialistame} we need the following technical result:

\begin{lemma}
\label{coroflang}
Let $f$ be analytic on a closed disc $\overline{D}(z_0, R), R>0$. Let $0 < R_1 < R$. Denote by $\|f\|_R$ the supremum norm of $f$ on the circle of radius $R$. Then for $z \in \overline{D}(z_0, R_1)$, we have:
$$|f^{(k)}(z)|\leq \frac{k! R}{(R-R_1)^{k+1}}\|f\|_R\,.$$
\end{lemma}

\begin{proof}
This lemma is an application of Cauchy's integral formula. For details~\cite{Lang99}, pp.~131.
\end{proof}

\begin{proof}[Proof of lemma~\ref{differentialistame}]~
\begin{enumerate}
\item $\frac{d}{dz}$ is linear.
\item To prove tameness we use lemma~\ref{coroflang}. Let $z\in A_n$ and choose $R= e^{-(n+1)}(e-1)$. Thus $\overline{D}(z, R)\subset A_{n+1}$. Hence $\|f\|_R\leq \|f\|_{n+1}$. 

In the notation of lemma~\ref{coroflang} we can use that $R_1=0$ and $k=1$  and calculate in this way: 
\begin{displaymath}
\|f'(z)\|_n\leq \frac{R}{R^{2}}\|f\|_{n+1}=\frac{e^{n+1}}{e-1}\|f\|_{n+1}\,.
\end{displaymath}

This description is independent of $z$. Thus 
\begin{displaymath}
\|f'\|_n\leq \frac{e^{n+1}}{e-1}\|f\|_{n+1}\,.
\end{displaymath}

Thus the differential is $(1,0,\frac{e^{n+1}}{e-1})$-tame.
\end{enumerate}
\end{proof}

A further class of spaces that is important for the description of twisted Kac-Moody algebras (compare definition~\ref{definitiontwistedloopalgebra}) are
spaces of holomorphic functions that satisfy some functional equation. We describe first the general setting and specialize then to the two most important cases, namely symmetric and antisymmetric holomorphic functions.

\begin{lemma}[Subspaces of $\textrm{Hol}(\mathbb{C}^*, \mathbb{C})$]
\label{subspacesofmg}
Let $k,l\in \mathbb{N}$ and $\omega=e^{\frac{2\pi i}{k}}$. The spaces 
\begin{displaymath}
\textrm{Hol}^{k,l}(\mathbb{C}^*, \mathbb C):=\{ f\in \textrm{Hol}(\mathbb{C}^*, \mathbb C)| f(\omega z)=\omega^{l}f(z)\}
\end{displaymath}
 are tame Fr\'echet spaces.
\end{lemma}

\begin{proof}
As usual for $f\in \textrm{Hol}^{k,l}(\mathbb{C}^*, \mathbb C)$, we put $\|f\|_n:=\displaystyle\sup_{z\in A_n}|f(z)|$. As  $\textrm{Hol}^{k,l}(\mathbb{C}^*, \mathbb C)$ is a closed subspace of $\textrm{Hol}(\mathbb{C}^*, \mathbb{C})$, it is a tame Fr\'echet space as a consequence of lemma~\ref{constructionoftamespaces}.
\end{proof}

Twisted affine Kac-Moody algebras arise as fixed point algebras of diagram automorphisms of non-twisted affine Kac-Moody algebras. The list of possible diagram automorphism shows that nontrivial diagram automorphisms have order $k=2$ or $k=3$~\cite{Carter05}. Thus the values of $k$ which are important for us are $k=2$ and $k=3$.

\noindent For $k=2$, lemma~\ref{subspacesofmg} has the corollaries:

\begin{corollary}[symmetric and antisymmetric loops]~
\label{holc*csaisfrechet}
\begin{itemize}
\item The space $\textrm{Hol}^s(\mathbb{C}^*, \mathbb C):=\{ f\in \textrm{Hol}(\mathbb{C}^*, \mathbb C), f(z)=f(-z)\}$ is a tame Fr\'echet space. 
\item The space $\textrm{Hol}^a(\mathbb{C}^*, \mathbb C):=\{ f\in \textrm{Hol}(\mathbb{C}^*, \mathbb C), f(z)=-f(-z)\}$ is a tame Fr\'echet space.
\end{itemize}
\end{corollary}

\noindent Lemmata~\ref{constructionoftamespaces} and~\ref{holc*cisfrechet} and corollary~\ref{holc*csaisfrechet} include the following result:

\begin{corollary}
\label{holc*cnisfrechet}
$F:= \textrm{Hol} (\mathbb C^*, \mathbb C^n)$, $F^s:= \textrm{Hol}^s (\mathbb C^*, \mathbb C^n)$ and $F^a:= \textrm{Hol}^a (\mathbb C^*, \mathbb C^n)$ are tame Fr\'echet spaces.
\end{corollary}

Let $V^n$ be a $n$-dimensional complex vector space. We want a tame structure on $\textrm{Hol}(\mathbb{C}^*, V^n)$. Using the corollary~\ref{holc*cnisfrechet} we get a tame Fr\'echet structure on the space $F:= \textrm{Hol} (\mathbb C^*, \mathbb C^n)$. This yields a tame structure on the spaces $\textrm{Hol}(\mathbb{C}^*, V^n)$, $\textrm{Hol}^s(\mathbb{C}^*, V^n)$ and $\textrm{Hol}^a(\mathbb{C}^*, V^n)$ only after the choice of an identification of $V^n$ with $\mathbb{C}^n$, hence a choice of a basis. As this construction uses this identification of $V^n$ with $\mathbb{C}^n$ we have to prove that the resulting tame structure is independent of it.

\noindent This is the content of the following lemma:

\begin{lemma}
Let $V^n$ be a $n$-dimensional complex vector space equipped with two norms $|\hspace{3pt}.\hspace{3pt}|$ and $|\hspace{3pt}.\hspace{3pt}|'$. Study the spaces $\textrm{Hol}^{k,l}(\mathbb{C}^*, V^n)$. Define gradings $\|f\|_n:=\displaystyle\sup_{z\in A_n}|f(z)|$ and $\|f\|_n':=\displaystyle\sup_{z\in A_n}|f(z)|'$. Those gradings are tame equivalent. 
\end{lemma}

\begin{proof}
Any two norms on a finite dimensional vector space are equivalent (see any book about elementary analysis, i.e.~\cite{Koenigsberger00}). Thus there exist constants $c_1$ and $c_2$ such that $|x|\leq c_1|x|'$ and $|x|'\leq c_2|x|$. Then $\|f\|_n:=\displaystyle\sup_{z\in A_n}|f(z)|\leq \displaystyle\sup_{z\in A_n}c_1|f(z)|'=c_1 \|f\|_n'$ and $\|f\|_n':=\displaystyle\sup_{z\in A_n}|f(z)|'\leq \displaystyle\sup_{z\in A_n}c_2|f(z)|' =c_2 \|f\|_n$. Thus they are tame equivalent. 
\end{proof}

\begin{corollary}
Any identification of  $\varphi:\mathbb{C}^n\longrightarrow V^n$ yields a $1$\ndash norm $\|v\|=\sum v_i$ on $V^n$. The tame structures induced on $\textrm{Hol}(\mathbb{C}^*, \mathbb{C})$ by two different choices are equivalent. Hence the tame structure on $V^n$ is independent of the identification of $V^n$ with $\mathbb{C}^n$.
\end{corollary}

As a consequence we have established the following result:

\begin{theorem}
Let $V^n$ be a complex vector space and $\|\phantom{z}\|$ any norm on $V^n$.  The space $\textrm{Hol}(\mathbb{C}^*, V^n)$ with the family of norms 
\begin{displaymath}
\|f\|_n=\sup_{z\in A_n}\|f(z)\|
\end{displaymath}
 is a tame Fr\'echet space. 
\end{theorem}

\begin{corollary}
Let $V^n$ be a complex vector space and $|\phantom{z}|_2$ the Euclidean norm on $V^n$.  The space $\textrm{Hol}(\mathbb{C}^*, V^n)$ with the family of norms 
\begin{displaymath}
\|f\|_n=\sup_{z\in A_n}|f(z)|_2
\end{displaymath}
is a tame Fr\'echet space. 
\end{corollary}

\section{Tame structures on loop algebras and loop groups}
\label{chap:tame}

In this chapter we introduce the functional analytic setting, which we use for the construction of affine Kac-Moody symmetric spaces in~\cite{Freyn12e}: We define Kac-Moody groups and Kac-Moody algebras of holomorphic loops and prove that they are tame Fr\'echet Lie groups resp. Lie algebras. Via the loop realization Kac-Moody groups (resp.\ algebras) can be viewed as  $2$-dimensional extensions of loop groups.  From a functional analytic point of view, a $2$-dimensional extension is unproblematic; the only exception is, that the action of group elements on the two dimensional extension has to be well-defined as well. Nevertheless remark that this extension plays a crucial role for geometric considerations. In consequence we focus our attention on the loop group (resp.\ algebra). In subsection~\ref{Lie_algebras_of_holomorphic_maps} we describe the loop algebras, in subsection~\ref{loopgroups} we turn our attention to the loop groups.

\subsection{Affine Kac-Moody algebras}
\label{subsect:affine_kac_moody}

In this section we recall some basic facts about affine Kac-Moody algebras and introduce the more general notion of geometric affine Kac-Moody algebras. A thorough study of the algebraic properties of geometric affine Kac-Moody algebras is done in~\cite{Freyn12b}.
The loop algebra realization of (algebraic) affine Kac-Moody algebras is developed in the books \cite{Kac90} and \cite{Carter05} from an algebraic point of view. We follow the more geometric approach to loop algebra realizations of Kac-Moody algebras, using the notion of~\cite{Terng95, Heintze09}. 

\noindent Let $\mathfrak{g}$ be a finite dimensional reductive Lie algebra over $\mathbb{F}=\mathbb{R}$ or $\mathbb{C}$. Hence by definition $\mathfrak{g}=\mathfrak{g}_{s}\oplus \mathfrak{g}_{a}$ is a direct product of a semisimple Lie algebra $\mathfrak{g}_s$ with an Abelian Lie algebra $\mathfrak{g}_{a}$. Let furthermore $\sigma=\sigma_{s}\otimes \sigma_{a}$ be some involution of $\mathfrak{g}$, such that $\sigma_s \in \textrm{Aut}(\mathfrak{g}_{s})$ denotes an automorphism of finite order of $\mathfrak{g}_s$ such that the restriction of $\sigma$ to any simple factor $\mathfrak{g}_{i}$ of $\mathfrak{g}$ is an automorphism of $\mathfrak{g}_{i}$ and $\sigma_{a}=\sigma|_{\mathfrak{g}_{a}}=\textrm{Id}$. If $\mathfrak{g}_{s}$ is a Lie algebra over $\mathbb{R}$ we assume $\mathfrak{g}_{s}$ to be a Lie algebra of compact type. We define the loop algebra $L(\mathfrak g, \sigma)$ as follows

\begin{displaymath}
L(\mathfrak g, \sigma):=\{f:\mathbb{R}\longrightarrow \mathfrak{g}\hspace{3pt}|f(t+2\pi)=\sigma f(t), f \textrm{ satisfies some regularity conditions}\}\label{abstractkacmoodyalgebra}\,.
\end{displaymath}

We use the notation $L(\mathfrak g, \sigma)$ to describe in a unified way constructions that can be realized explicitly for loop algebras satisfying various regularity conditions. Regularity conditions used in applications include the following:
\begin{itemize}
\item $H^0$-Sobolev loops, denoted $L^0\mathfrak{g}$,
\item smooth, denoted $L^{\infty}\mathfrak{g}$,
\item real analytic, denoted $L_{\textrm{an}}\mathfrak{g}$,
\item (after complexification of the domain of definition)  holomorphic on $\mathbb{C}^*$, denoted $M\mathfrak{g}$,
\item holomorphic on an annulus $A_n\subset \mathbb{C}$, denoted $A_n\mathfrak{g}$, or 
\item algebraic (or equivalently: with a finite Fourier expansion), denoted $L_{alg}\mathfrak{g}$.
\end{itemize}

\begin{definition}[affine Kac-Moody algebra]
\label{geometricaffinekacmoodyalgebra}
The indecomposable geometric affine Kac-Moody algebra associated to a pair $(\mathfrak{g}, \sigma)$ as described above is the algebra:
$$\widehat{L}(\mathfrak g, \sigma):=L(\mathfrak g, \sigma) \oplus \mathbb{F}c \oplus \mathbb{F}d\,,$$
equipped with the lie bracket defined by:
\begin{alignat*}{1}
  [d,f(t)]&:=f'(t)\, ,\\
  [c,c]=[d,d]&:=0\, ,\\
	[c,d]=[c,f(t)]&:=0\, ,\\
  [f,g](t)&:=[f(t),g(t)]_{0} + \omega\left(f(t),g(t)\right)c\,.
\end{alignat*}

Here $f\in L(\mathfrak g, \sigma)$ and $\omega$ is a certain antisymmetric $2$-form on $M\mathfrak g$, satisfying the cocycle condition. If the regularity of $L(\mathfrak g, \sigma)$ is chosen such that $L(\mathfrak g, \sigma)$ contains non-differentiable functions, then $d$ is defined on the (dense) subspace of differentiable functions.
\end{definition}

Explicit realizations using different regularity conditions are common. We use holomorphic or algebraic loops; for these we can define \begin{displaymath}\omega(f,g):=\textrm{Res}(\langle f, g' \rangle)\,.\end{displaymath}

Let us reformulate the definition for non twisted affine Kac-Moody algebras in terms of functions on $\mathbb{C}^*$: Assume $\sigma=\textrm{Id}$. First we develop a function $f\in L(\mathfrak{g}, \textrm{Id})$ into its Fourier series $f(t)=\sum a_n{e}^{int}$. Then this function is naturally defined on a circle $S^1$; we understand this circle to be embedded as the  unit circle $\{z\in \mathbb{C}^*||z|=1\}\subset\mathbb{C}^*$; in this way the parameter $t$ gets replaced by the complex parameter $z:=e^{it}$, with $|z|=1$; understanding the Fourier expansion now as a Laurent expansion of $F$, we can calculate the annulus, $A_n$, on which this series is defined. For example the holomorphic realization $M\mathfrak{g}$, is defined by the condition, that for any $f\in M\mathfrak{g}$ the Fourier expansion describes a Laurent series expansion of a holomorphic function on $\mathbb{C}^*$.

\begin{definition}
The complex geometric affine Kac-Moody algebra associated to a pair $(\mathfrak{g}, \sigma)$ is the algebra
$$\widehat{M\mathfrak g}:=M\mathfrak g \oplus \mathbb{C}c \oplus \mathbb{C}d\,,$$
equipped with the lie bracket defined by:
\begin{alignat*}{1}
  [d,f(z)]&:=izf'(z)\,,\\
  [c,c]=[d,d]&:=0\,,\\
  [c,d]=[c,f(z)]&:=0\,,\\
  [f,g](z)&:=[f(z),g(z)]_{0} + \omega(f(z),g(z))c\,.
\end{alignat*}
\end{definition}

\noindent As $\frac{d}{dt}e^{itn}=ine^{itn}=inz^n=iz\frac{d}{dz}z^n$ both definitions coincide.

\begin{definition}[semisimple geometric affine Kac-Moody algebra]

A geometric affine Kac-Moody algebra $\widehat{L}(\mathfrak{g}, \sigma)$ is called 
\begin{itemize}
\item semisimple if $\mathfrak{g}$ is semisimple,
\item simple if $\mathfrak{g}$ is simple.
\end{itemize}
\end{definition}

%%%%%%%%%%%%%%%%%%%%%%%%%%%%%%%%%%%%%%%%%%%%%%%%%%%%%%%%%%%%%%%%%%%%%%%%%%%%%%%%%%%%%%%%%%%%%%%%%%%%%%%%%%%%%%%%%%%%%%%%%%%%%%%%%%%%%%%%%%%%%%%%%%%%%%%%%%%%%%%%%%%%%%%%%%%%%%%%%%%%%%%%%%%%%%%%%%%%%%%%%%%%%%%%%%%%%%%%%%%%%%%%%%%%%%%%%%%%%%%%%%%%%%%%%%%%%%%%%%%%%%%%%%%%%%%%%%%%%%%%%%%%%%%%%%%%%%%%%%%%%%%%%%%%%%%%%%%%%%%%%%%%%%%%%%%%%%%%%%%%%%%%%%%%%%%%%%%%%%%%%%%%%%%%%%%%%%%%%%%%%%%%%%%%%%%%%%%%%%%%%%%%%%%%%%%%%%%%%%%%%%%%%%%%%%%%%%%%%%%%%%%%%%%%%%%%%%%%%%%%%%%%%%%%%%%%%%%%%%%%%%%%%%%%%%%%%%%%%%%%%%%%%%%%%%%%%%%%%%%%%%%%%%%%%%%%%%%%%%%%%%%%%%%%%%%%%%%%%%%%%%%%%%%%%%%%%%%%%%%%%%%%%%%%%%%%%%%%%%%%%%%%%%%%%%%%%%%%%%%%%%%%%%%%%%%%%%%%%%%%%%%%%%%%%%%%%%%%%%%%%%%%%%%%%%%%%%%%%%%%%%%%%%%%%%%%%%%%%%%%%%%%%%%%%%%%%%%%%%%%%%%%%%%%%%%%%%%%%%%%%%%%%%%%%%%%%%%%%%%%%%%%%%%%%%%%%%%%%%%%%%%%%%%%%%%%%%%%%%%%%%%%%%%%%%%%%%%%%%%%%%%%%%%%%%%%%%%%%%%%%%%%%%%%%%%%%%%%%%%%%%%%%%%%%%%%%%%%%%%%%%%%%%%%%%%%%%%%%%%%%%%%%%%%%%%%%%%%%%%%%%%%%%%%%%%%%%%%%%%%%%%%%%%%%%%%%%%%%%%%%%%%%%%%%%%%%%%%%%%%%%%%%%%%%%%%%%%%%%%%%%%%%%%%%%%%%%%%%%%%%%%%%%%%%%%%%%%%%%%%%%%%%%%%%%%%%%%%%%%%%%%%%%%%%%%%%%%%%%%%%%%%%%%%%%%%%%%%%%%%%%%%%%%%%%%%%%%%%%%%%%%%%%%%%%%%%%%%%%%%%%%%%%%%%%%%%%%%%%%%%%%%%%%%%%%%%%%%%%%%%%%%%%%%%%%%%%%%%%%%%%%%%%%%%%%%%%%%%%%%%%%%%%%%%%%%%%%%%%%%%%%%%%%%%%%%%%%%%%%%%%%%%%%%%%%%%%%%%%%%%%%%%%%%%%%%%%%%%%%%%

\subsection{Lie algebras of holomorphic maps}
\label{Lie_algebras_of_holomorphic_maps}

Let $\mathfrak{g}$ be a complex reductive Lie algebra that is a direct product of simple Lie algebras with an Abelian Lie algebra.  Simple Lie algebras are completely classified by their root systems; a complete list is given as follows: 
\begin{displaymath}
A_n, B_{n, n\geq 2}, C_{n, n\geq 3}, D_{n, n \geq 4}, E_6, E_7, E_8, F_4, G_2\,. 
\end{displaymath}

Abelian Lie algebras are classified by their dimension. A complex reductive Lie algebra $\mathfrak{g}_{\mathbb{C}}$ has an up to conjugation unique compact real form $\mathfrak{g}_c$. This compact real form is defined as the real reductive Lie algebra $\mathfrak{g}_{\mathbb{R}}$ such that $\mathfrak{g}$ is a direct product of the (up to conjugation) unique compact real forms of the simple factors of $\mathfrak{g}_{\mathbb{C}}$ together with a compact real form of the Abelian factor. A compact real Abelian Lie algebra of dimension $n$ is just $\mathbb{R}^n$ together with the trivial bracket; but we define the exponential function such that the exponential image is a torus. Hence a compact real Lie algebra can be identified with the purely imaginary part of a complex Abelian Lie algebra.

\begin{definition}[complex holomorphic non-twisted Loop algebra]
\label{complex holomorphic non-twisted Loop algebra}
Let $\mathfrak{g}_{\mathbb C}$ be a finite-dimensional reductive complex Lie algebra. 

\begin{enumerate}
\item  The loop algebra $A_n\mathfrak{g}_{\mathbb{C}}$ is the vector space
$$A_n\mathfrak{g}_{\mathbb{C}}:=\bigcup_{A_n\subset U \textrm{open}}\{f:U \longrightarrow
\mathfrak{g}_\mathbb {C}| \textrm{ f is holomorphic}\}\,,$$
equipped with the natural Lie bracket:
$$[f,g]_{L_n\mathfrak{g}}(z):=[f,g]_0(z):=[f(z),g(z)]_{\mathfrak{g}}\,.$$

\item  The loop algebra $M\mathfrak{g}_{\mathbb{C}}$ is the vector space
$$M\mathfrak{g}_{\mathbb{C}}:=\{f:\mathbb C^* \longrightarrow
\mathfrak{g}_\mathbb {C}| \textrm{ f is holomorphic}\}\,,$$
equipped with the natural Lie bracket:
$$[f,g]_{M\mathfrak{g}}(z):=[f,g]_0(z):=[f(z),g(z)]_{\mathfrak{g}}\,.$$
\end{enumerate}
\end{definition}

\begin{lemma}~
\label{mgfrechetagbanach}
\begin{enumerate}
\item $M\mathfrak{g}_{\mathbb{C}}$ is a tame  space.
\item $A_n\mathfrak{g}_{\mathbb{C}}$ is a Banach space. 
\end{enumerate}
\end{lemma}

\begin{proof}
The first assertion is a consequence of corollary~\ref{holc*cnisfrechet}. The second assertion is a consequence of Montel's theorem stating that absolute convergent sequences of holomorphic functions converge to a holomorphic function~\cite{Berenstein91}. 
\end{proof}

\noindent The inclusions $S^1 =A_0 \subset \dots A_n \subset A_{n+1} \subset \dots \subset \mathbb{C}^*$ induce the reversed inclusions on the associated loop algebras:

\begin{displaymath}
M\mathfrak{g}_{_{\mathbb{C}}} \subset \dots \subset A_{n+1}\mathfrak{g}_{\mathbb{C}}\subset A_n\mathfrak{g}_{\mathbb{C}} \subset \dots \subset A_0\mathfrak{g}_{\mathbb{C}}= L_{\textrm{hol}}\mathfrak{g}_{\mathbb{C}}\,.
\end{displaymath}

$L_{\textrm{hol}}\mathfrak{g}_{\mathbb{C}}$ denotes functions holomorphic in a small open set around $S^1\subset \mathbb{C}^*$.

To describe the twisted loop algebras we recall the graph automorphisms of the finite dimensional simple Lie algebras:
the following list contains the simple algebras $A$ with a nontrivial diagram automorphism $\sigma$ and the type of the fixed point algebra (compare \cite{Carter05}). 
\[
\begin{array}{lrccccc}
A&:&A_{2k}&A_{2k+1}&D_{k+1}&D_4&E_6\\
\textrm{Order of }\sigma&:&2&2&2&3&2\\
A^1&:&B_k&C_k&B_k&G_2&F_4
\end{array}
\]

\begin{definition}[(twisted) loop algebra, $\textrm{ord}(\sigma ) =2$]
\label{definitiontwistedloopalgebra}
Let $\mathfrak{g}_{\mathbb C}$ be a finite dimensional semisimple complex Lie
algebra of type $A_{k}$, $D_{k, k \geq 5}$ or $E_6$,  $\sigma$ the diagram automorphism. Let $\mathfrak{g}_{\mathbb C}:=\mathfrak{g}_{\mathbb C}^1 \oplus \mathfrak{g}_{\mathbb C}^{-1}$ be the decomposition into the $\pm$-eigenspaces of $\sigma$. Let $X\in \{A_n, \mathbb{C}^*\}$. If $X=A_n$ holomorphic functions on $X$ are understood to be holomorphic on an open set containing $X$.

\noindent Then the loop algebra $(X\mathfrak{g})^{\sigma}$ is the vector space
$$X\mathfrak{g}^{\sigma}:=\{ f\in X\mathfrak{g}| f(-z)=\sigma(f(z))\}\,,$$
equipped with the natural Lie bracket:
$$[f,g]_{X\mathfrak{g}^{\sigma}}(z):=[f,g]_0(z):=[f(z),g(z)]_{\mathfrak{g}}\,.$$

\end{definition}

\begin{remark}[(twisted) loop algebra, $\textrm{ord}(\sigma ) =3$]
For the algebra of type $D_4$ there exists an automorphism $\sigma$ of order $3$. In this case we get exactly the same results as for the other types. The main difference is that we have three eigenspaces, corresponding to $\{\omega, \omega^2, \omega^3=1\}$ for $\omega= e^{\frac{2\pi i }{3}}$. For a function $f$ in the loop algebra $M\mathfrak{g}$, this results in a twisting condition $f(\omega z)=\sigma f(z)$ (for details compare again~\cite{Carter05}).
\end{remark}

\begin{lemma}[Banach- and Fr\'echet structures on twisted loop algebras]~
\begin{enumerate}
	\item $A_n\mathfrak{g}^{\sigma}$ equipped with the norm $\|\hspace{3pt} \|_n$ is a Banach Lie algebra,
	\item $M\mathfrak{g}^{\sigma}$ equipped with the norms $\|\hspace{3pt} \|_n$ is a tame Fr\'echet Lie algebra. 
\end{enumerate}
\end{lemma}

\begin{proof}
Closed subspaces of Banach spaces are Banach spaces and closed subspaces of tame Fr\'echet spaces are tame Fr\'echet spaces (lemma~\ref{constructionoftamespaces}). 
\end{proof}

\begin{theorem}
 The following topologies on $M\mathfrak{g}_{\mathbb{C}}$ are equivalent:
\begin{enumerate}
 \item the compact-open topology,
 \item the topology of compact convergence,
 \item the Fr\'echet topology,
\end{enumerate}
\end{theorem}

Recall that on a Fr\'echet space $(F, \|\ \|_n)$
\begin{displaymath}
 d(f,g)=\sum_n\frac{1}{2^n}\textrm{\small$\frac{\|f-g\|_n}{1+\|f-g\|_n}$\normalsize}
\end{displaymath}
defines a metric. The Fr\'echet topology coincides with the topology generated by the metric $d(f,g)$ (see~\cite{Hamilton82}).

\begin{proof} We prove $(1)\Leftrightarrow (2)$, $(3)\Leftrightarrow (2)$.
 \begin{enumerate}
  \item [$(1)\Leftrightarrow$]$ (2)$ The equivalence between the compact-open topology and the topology of compact convergence is well-known for spaces of continuous functions $C(X,Y)$ for $X$, $Y$ metric spaces. The extension to the setting of holomorphic functions is a consequence of Montel's theorem.
\item [$(3)\Leftrightarrow$]$ (2)$ Let $(f_k) \subset M\mathfrak{g}$ be a sequence converging to $f_0$ in the topology of compact convergence. Then for every $n_0\in \mathbb{N}$ and $\epsilon> 0$ there is a $k_0$ such that for all $k\geq k_0$ the estimate $\|f_k-f_0\|_n\leq\epsilon$ is satisfied for all $n\leq n_0$. Hence
\begin{align*}
 d(f_k,f_0)& =\sum_{n=0}^{\infty}\frac{1}{2^n} \textrm{\small$\frac{\|f_k-f_0\|_n}{1+\|f_k-f_0\|_n}$\normalsize}=\\
           & =\sum_{n=0}^{n_0}\frac{1}{2^n}\textrm{\small$\frac{\|f_k-f_0\|_n}{1+\|f_k-f_0\|_n}$\normalsize}+
           \sum_{n=n_0+1}^{\infty}\frac{1}{2^n}\textrm{\small$\frac{\|f_k-f_0\|_n}{1+\|f_k-f_0\|_n}$\normalsize}\leq\\
&\leq \sum_{n=0}^{n_0}\frac{1}{2^n}\frac{\epsilon}{1+\epsilon}+ \sum_{n=n_0+1}^{\infty}\frac{1}{2^n}\leq\\
&\leq 2\epsilon+\left(\frac{1}{2}\right)^{n_0}\, .
\end{align*}
Hence $(f_k) \subset M\mathfrak{g}$ converges in the Fr\'echet topology.
Conversely let $(f_k) \subset M\mathfrak{g}$ be a sequence converging to $f_0$ in the Fr\'echet topology. Then we have
\begin{displaymath}
\lim_{k\rightarrow \infty} d(f_k,f_0)=\lim_{k\rightarrow\infty} \sum_n\frac{1}{2^n}\textrm{\small$\frac{\|f_k-f_0\|_n}{1+\|f_f-f_0\|_n}$\normalsize}=0\, .
\end{displaymath}
As $d(f_k,f_0)\geq \sup_n \frac{1}{2^n}\frac{\|f_k-f_0\|_n}{1+\|f_f-f_0\|_n}$ we conclude that 
\begin{displaymath}
 \lim_{k\rightarrow \infty}\sup_n \frac{1}{2^n}\frac{\|f_k-f_0\|_n}{1+\|f_f-f_0\|_n}=0\, .
\end{displaymath}
As for any compact set $K\subset \mathbb{C}^*$ there is some $n$ such that $K\subset A_n$, this yields compact convergence.
\end{enumerate}

\end{proof}

The adjoint action $ad(g):M\mathfrak{g}^{\sigma}\longrightarrow M\mathfrak{g}^{\sigma}$ is $(0,0, 2\|g\|_n)$-tame for each  $g\in M\mathfrak{g}^{\sigma}$.  Contrast this with the situation for affine Kac-Moody algebras described in section~\ref{holomorphicstructuresonkacmoodyalgebras}.

\noindent Having described the holomorphic complex loop algebras which we need, we turn now to some objects derived from them, namely the compact real forms and spaces of differential forms.

\noindent We start with real forms of compact type:

\begin{definition}[compact real form of a holomorphic non-twisted loop algebra]
Let $\mathfrak{g}_{\mathbb C}$ be a finite-dimensional semisimple complex Lie algebra and $\mathfrak{g}$ its compact real form.
 The loop algebra $X\mathfrak{g}_{\mathbb{R}}^{\sigma}$ is the vector space
\begin{displaymath}X\mathfrak{g}_{\mathbb{R}}^{\sigma}:=\{f\in X\mathfrak{g}_{\mathbb{C}}^{\sigma}| f(S^1)\subset \mathfrak{g} \}\,,\end{displaymath}
equipped with the natural Lie bracket:
\begin{displaymath}[f,g]_{X\mathfrak{g}}(z):=[f,g]_0(z):=[f(z),g(z)]_{\mathfrak{g}}\,.\end{displaymath}
\end{definition}

\noindent As a holomorphic function on $X$ can be expanded into its Laurent
series, one can represent every element of a loop algebra by a series 
\begin{displaymath} f(z):= \sum_{n} g_n z^n\end{displaymath}
with $g_n \in \mathfrak{g}$.

\begin{lemma}
The condition $f(S^1)\subset \mathfrak{g}_{\mathbb{R}}$  is equivalent to the condition $g_n=-\bar{g}_{-n}^{t}$.
\end{lemma}

\begin{proof}
Let $z=e^{it}\in S^1\subset \mathbb{C}^*$ and let $g_n=g_n^{r}+ig_n^{i}$ be the decomposition of $g_n$ into its real and imaginary parts. Then we find
\begin{align*}
f(z)&= \sum_{n\in \mathbb{Z}} g_n z^n=\sum_{n\in \mathbb{Z}} g_n e^{itn}=\\
&=a_0+\sum_{n\in \mathbb{N}}g_n e^{itn}+g_{-n}e^{-itn}=\\
&=a_0+\sum_{n\in \mathbb{N}}(g_n^{r}+ig_n^{i}) (\cos(tn)+i\sin(tn))+(g_{-n}^{r}+ig_{-n}^{i})(\cos(-tn)+i\sin(-tn))=\\
&=a_0+\sum_{n\in \mathbb{N}}(g_n^{r}\cos(tn)-g_n^{i}\sin(tn)) +i(g_{n}^{i}\cos(tn)+g_n^r\sin(tn))+\\
&\qquad \qquad(g_{-n}^{r}\cos(-tn)- g_{-n}^{i}\sin(-tn) +i(g_{-n}^{i}\cos(-tn)+(g_{-n}^{r}\sin(-tn))=\\
&=a_0+\sum_{n\in \mathbb{N}}(g_n^r+g_{-n}^r+i(g_{n}^{i}+g_{-n}^{i}))\cos(tn)+(-g_n^{i}+g_{-n}^{i}+i(g_n^r-g_{-n}^r))\sin(tn)
\end{align*}
Now $x\in \mathfrak{g}_{\mathbb{C}}$ is in $\mathfrak{g}$ if $x=-\overline{x}^t$. This gives for the coefficients $g_{n}^r$:
\begin{align*}
g_n^r+g_{-n}^r&=-(g_n^r)^t-(g_{-n}^r)^t\, ,\\
g_n^r-g_{-n}^r&=(g_n^r)^t-(g_{-n}^r)^t\, .
\end{align*}
Adding both we get $g_n^{r}=-(g_{-n}^r)^t$.
In a similar way we get for the imaginary parts of the coefficients $g_n^{i}=(g_{-n}^i)^t$ and thus the result.
\end{proof}

\begin{lemma}
$M\mathfrak{g}_{\mathbb{R}}$ is a tame Fr\'echet space.
\end{lemma}

\begin{proof}
$M\mathfrak{g}_{\mathbb{R}}\subset M\mathfrak{g}_{\mathbb{C}}$ is a closed subspace and thus tame according to lemma~\ref{constructionoftamespaces}.
\end{proof}

\begin{definition}~
\begin{enumerate}
\item $\Omega^1(X,\mathfrak{g}_{\mathbb{C}})$ is the space of $\mathfrak{g}_{\mathbb{C}}$-valued $1$-forms on $X$; elements $\omega\in \Omega^1(X,\mathfrak{g}_{\mathbb{C}})$ are of the form $\omega(z)=f(z)dz$ with $f(z)\in X\mathfrak{g}_{\mathbb{C}}$. We define a family of norms  by $\|\omega\|_n:=|f(z)|_n$.
\item $\Omega^1(X,\mathfrak{g}_{\mathbb{C}})_\mathbb{R}$ is the space of $\mathfrak{g}_{\mathbb{C}}$-valued $1$-forms on $X$ such that $f(S^1) \subset \mathfrak{g}_{\mathbb{R}}$.
\end{enumerate}
\end{definition}

\noindent As $M\mathfrak{g}_{\mathbb{C}}$ and $M\mathfrak{g}_{\mathbb{R}}$ are  tame Fr\'echet spaces, also $\Omega^1(X,\mathfrak{g}_{\mathbb{C}})$ and  $\Omega^1(X,\mathfrak{g}_{\mathbb{C}})_\mathbb{R}$ are tame. Remark, that this is not the topology as a dual space.

\begin{remark}
Real forms of the algebras $X\mathfrak{g}^{\sigma}_{\mathbb{C}}$ correspond to conjugate-linear involutions of $X\mathfrak{g}^{\sigma}_{\mathbb{C}}$: assign to a real form the conjugation with respect to it. In the other direction, fixed point algebras of conjugate-linear involutions are real forms.
Hence, real forms are closed subalgebras. Thus by an application of lemma~\ref{constructionoftamespaces} real forms of $M\mathfrak{g}_{\mathbb{C}}^{\sigma}$ are tame, real forms of $A_n{\mathfrak{g}}_{\mathbb{C}}^{\sigma}$ are Banach. 
\end{remark}

\subsection{Lie groups of holomorphic maps}
\label{loopgroups}

Up to now we studied analytic structures on loop algebras but not on the
associated loop groups. In short the main result is that all loop algebras interesting to us are tame Lie algebras. In this section we prove similar results for loop groups. Let $G$ be a compact semisimple Lie group and $G_{\mathbb{C}}$ its complexification.

\subsubsection{Foundations}

Let us recall the definition of a smooth tame Lie group from~\cite{Hamilton82}:

\begin{definition}[smooth tame Lie group]
A smooth tame Lie group is a smooth tame Fr\'echet manifold $G$ equipped with a
group structure such that the multiplication map 
\begin{displaymath}
\varphi: G\times G\longrightarrow G,\quad (g,h)\mapsto gh
\end{displaymath}
and the inverse map
\begin{displaymath}
\varphi: G\longrightarrow G,\quad g\mapsto g^{-1}
\end{displaymath}
are smooth tame maps.
\end{definition}

In this section we show that the following groups are smooth tame Lie groups:

\begin{definition}[complex loop groups]~ 
\label{A_nGC}
\begin{enumerate}
\item The loop group $A_nG$ is the group
\begin{displaymath}A_nG_{\mathbb C}:=\{f:A_n\longrightarrow G_{\mathbb C}|
\textrm{ f is holomorphic}\}\,.\end{displaymath}
The multiplication is defined to be $fg(z):= f(z)g(z)$ for $f,g \in A_nG$. 

\item The loop group $MG$ is the group 
$$MG_{\mathbb C}:=\{f:\mathbb C^*\longrightarrow G_{\mathbb C}|
\textrm{ f is holomorphic}\}\,.$$
The multiplication is defined to be $(fg)(z):= f(z)g(z)$ for $f,g \in MG$.
\end{enumerate}
\end{definition}

\begin{definition}[real form of the compact type]~
\label{A_nGR}
\begin{enumerate}
\item The real form of the compact type $A_nG_{\mathbb{R}}$ is defined to be 
$$A_nG_{\mathbb R}:=\{f\in A_nG_{\mathbb C}| f(S^1) \subset G_{\mathbb R}\}\,.$$
\item The real form of the compact type $MG_{\mathbb{R}}$ is defined to be 
$$MG_{\mathbb R}:=\{f\in MG_{\mathbb C}| f(S^1) \subset G_{\mathbb R}\}\,.$$
\end{enumerate}
\end{definition}

There are exponential functions 
\begin{align*}A_n\exp: A_n\mathfrak{g}&\longrightarrow A_nG\quad \textrm{and}\\
\Mexp: M\mathfrak{g}&\longrightarrow MG,
\end{align*} 
defined pointwise using the group exponential function $\exp: \mathfrak{g}\rightarrow G$:
$$
\begin{aligned}
[(A_n\exp)(f)](z)&:=\exp(f(z))\,,\\
[(\Mexp)(f)](z)&:=\exp(f(z))\,.
\end{aligned}$$

Let us remark, that for any $z\in \mathbb{C}^*$ resp. $z\in A_n$, the curve  $\gamma(t):=[(\Mexp)(t\cdot f)](z)$ resp. $\gamma(t):=[(A_n\exp)(t\cdot f)]$ defines $1$-parameter subgroups.

The next important object needed to describe the connection between the loop algebras and the loop groups is the definition of the Adjoint action $\textrm{Ad}$:

As usual it is defined pointwise using the Adjoint action of the Lie group $G_{\mathbb{K}}$, $\mathbb{K}\in \{\mathbb{R},\mathbb{C}\}$:

$$
\begin{array}{ll}
(\textrm{Ad}(A_nG)_{\mathbb{K}} \times A_n\mathfrak{g}_{\mathbb{K}})\longrightarrow A_n\mathfrak{g}_{\mathbb{K}}, \hspace{10pt}&(f,h) \mapsto fhf^{-1}\,,\\
(\textrm{Ad}(MG)_{\mathbb{K}}\times M\mathfrak{g}_{\mathbb{K}})\longrightarrow M\mathfrak{g}_{\mathbb{K}}, &(f,h) \mapsto fhf^{-1}\,,\\
\end{array}
$$
where $$fhf^{-1}(z):= f(z)h(z)f^{-1}(z)\simeq \textrm{Ad}(f(z)) (h(z))\, .$$

For the Adjoint action of groups of the compact type to be well-defined we have to check, that the condition $f(S^1)\subset \mathfrak{g}_{\mathbb{R}}$ is preserved. This is a consequence of the adjoint action for finite dimensional compact Lie groups: for all $z\in S^1$ we have $f(z)\in G_{\mathbb{R}}$ and $h(z)\in \mathfrak{g}_{\mathbb{R}}$. Thus  the condition $\textrm{Ad}(f)h(z)\in \mathfrak{g}_{\mathbb{R}}$ is preserved pointwise.

\begin{lemma}
The exponential function and the Adjoint action satisfy the identity: 
\begin{displaymath}\textrm{Ad}\circ \exp X= e^{ad}X\quad \textrm{for} \quad X\in \{A_n, \mathbb{C}^*\}\, .\end{displaymath}
\end{lemma}

\begin{proof}
Applying the well-known identity for finite dimensional Lie algebras (resp. Lie groups) we get that the identity is valid pointwise.\end{proof}

\noindent We now investigate the functional analytic nature of the groups $A_nG$ and $MG$: to fix some notation let $X\sigma$ denote by abuse of notation an involution of $X\mathfrak{g}$ resp.\ of $XG$. Let $XG_D$ denote a real form of non-compact type of $XG$, and denote by $\textrm{Fix}(X\sigma)$ the fixed point group of an involution $X\sigma$.

E.\ Heintze and C.\ Gro\ss~\cite{Heintze09} show that real forms of the non-compact type of a complex simple Kac-Moody algebra are in bijection with involutions of the compact real form (which is unique up to conjugation).  Let $X\mathfrak{g}_{\mathbb{R}}$ be a compact real form with involution $X\sigma$. We denote by $X\mathfrak{g}_{D, \sigma}$ the real form of non-compact type associated to $X\sigma$.

Let us focus on a description of the groups $A_nG$: as $A_n$ is compact we can follow the classical strategy to define manifold and Lie group structures.  We start by defining a chart on an open set containing the identity with values in the Lie algebra via the exponential map; then we use left translation to construct an atlas of the whole group. This strategy yields the following basic results:

\begin{theorem}~
\begin{enumerate}
\item $A_nG_{\mathbb{R}}$ and $A_nG_{\mathbb{C}}$ are  Banach-Lie groups.
\item Real forms $A_nG_D$ of non-compact type of $A_nG_{\mathbb{C}}$ are Banach-Lie groups.
\item Quotients $A_nG_{\mathbb{R}}/\textrm{Fix}(A_n\sigma)$ and  $A_nG_{D, \sigma}/\textrm{Fix}(A_n\sigma)$ are Banach manifolds. 
\end{enumerate}
\end{theorem}

 For Banach-Lie groups and Banach manifolds, there exists a huge literature; for a classical introduction see for example~\cite{Palais68}. 

\noindent For the groups $MG$ themselves the theory is considerably more difficult. The crucial observation is the fact that the exponential map has in general not to be a local diffeomorphism.

\noindent We give an example of this strange phenomenon:

\begin{example}[$MSL(2,\mathbb C)$]
\label{sl2chasnodiffeomorphicexponentialmap}
We study the Lie group $SL(2,\mathbb{C})$. As is well known, $$\exp: \mathfrak{sl}(2,\mathbb {C})\longrightarrow SL(2,\mathbb{C})$$ is not surjective. For example elements $g\in SL(2, \mathbb{C})$ conjugate to the element \tiny$\left(
\begin{array}{cc}
  -1 & 1\\
  0 & -1
\end{array}
\right)$ \normalsize
are not in the image of $\exp \left(\mathfrak{sl}(2,\mathbb C)\right)$ (see\cite{DuistermaatKolk00}).

We want to show that there exists a sequence $f_n \in MSL(2,\mathbb C)$ which
converges to the identity element 
\tiny$\left(
\Id(z):=\begin{array}{cc}
  1 & 0\\
  0 & 1
\end{array}
\right)$ \normalsize 
in the compact-open-(tame Fr\'echet) topology but is not contained in the image of $\Mexp$. To this end, we have to construct $f_n$ in a way that it contains points that are not in the image of $\exp
\left(\mathfrak{sl}(2,\mathbb C)\right)$.
We define
\begin{align*}
f_n(z)&=
\left(
	\begin{array}{cc}e^{\pi z/n} & -i z/n\\0&e^{- \pi z/n}\end{array}
\right)\,.\\
\intertext{Then}
f_n(in)&=
\left(
	\begin{array}{cc}-1 & 1\\0&-1\end{array}
\right)\,.
\end{align*}
In consequence $f_n$ is not contained in $\Im(\Mexp(\mathfrak{sl}(2, \mathbb C)))$ for any $n\in \mathbb{N}$. On the other hand, for any fixed $z_0 \in \mathbb C^*$ we find the limit
$$\lim_{n\rightarrow \infty}f_n(z_0)=
\left(	\begin{array}{cc}1 & 0\\0&1\end{array}
\right)=\textrm{Id}\,.$$

Hence in the compact-open topology for every neighborhood $U_k$ of the identity there exist $n_k\in \mathbb N$ such that $\forall n\geq n_k: f_n \in U_k$.

This oberservation concludes the proof that $f_n$ is not a local diffeomorphism.
\end{example}

\noindent Especially this observation contains the corollary that the groups $MG$ are no locally exponential Lie groups in the sense of Karl-Hermann Neeb~\cite{Neeb06}, that is Lie groups such that $\exp$ is a local diffeomorphism. 

Hence, we have to find another way to define manifold structures on $MG$.
We start by describing some results about the relationship between $MG$ and $M\mathfrak g$. Then we show that loop groups satisfy the weaker axioms for pairs of exponential type introduced by Hideki Omori.

\begin{definition}
The tangential space $T_p(MG)$ is defined as the space of path-equivalence classes of smooth paths. Hence two curves
$$\gamma_i:(-\epsilon_i, \epsilon_i)\times \mathbb{C}^*\longrightarrow G_{\mathbb{C}},\ i=1,2$$ are equivalent if there is some $\epsilon_0>0$ such that 
$$\gamma_1|_{(-\epsilon_0, \epsilon_0)\times \mathbb{C}^*}=\gamma_2|_{(-\epsilon_0, \epsilon_0)\times \mathbb{C}^*}\,.$$
\end{definition}

A path depends holomorphically on the second factor and smoothly on the first factor. Let us remark, that one can use weaker regularity conditions on the first factor (i.e. $C^k$-dependence); we will not pursue the study of these weaker regularity conditions any further. In contrast the holomorphic dependence on the second factor is derived from the regularity condition imposed on the (holomorphic) Kac-Moody group.

\noindent The relationship between $M\mathfrak{g}$ and $MG$ is described by
the following three results (we postpone the proofs after the discussion):

\begin{theorem}[Tangential space]
\label{tangential space}
Let $\mathfrak{g}$ be the Lie algebra of $G$. Then
$$M\mathfrak{g}= T_e(MG)\,.$$
\end{theorem}

\noindent Moreover the tangential space $T_e(MG)$ is isomorphic to the Lie algebra of left-invariant vector fields on $MG$. 

\noindent While the tangential space of a loop group is the corresponding loop algebra, the exponential map behaves in general badly.  More precisely, we have the following result:

\begin{theorem}[Loop groups whose exponential map is no local diffeomorphism]
\label{expnodiffeomorphism}
Let $G_{\mathbb{C}}$ be a complex semisimple Lie group. 
$$\Mexp:M\mathfrak{g}\longrightarrow MG$$
is not a local diffeomorphism.
\end{theorem}

\noindent This is in sharp contrast to the case of nilpotent Lie groups For these groups we have the following result:

\begin{proposition}[Loop groups whose exponential map is a local diffeomorphism]
\label{diffeomorphicexponentialmap}
Let $G_{\mathbb{C}}$ be a complex Lie group such that its universal cover is biholomorphically equivalent to $\mathbb{C}^n$. Then its exponential map $\Mexp$ is a local diffeomorphism.
\end{proposition}

\begin{corollary}
\label{examplesofliegroupswithdiffexp}
Let $G$ be a complex Lie group. If $G$ is nilpotent (i.e.\ Abelian) then $\Mexp$ is a diffeomorphism.
\end{corollary}

\noindent Corollary~\ref{examplesofliegroupswithdiffexp} is a direct consequence of proposition~\ref{diffeomorphicexponentialmap}, as Abelian and nilpotent Lie groups have $\mathbb{C}^n$ as their universal cover~\cite{Knapp96}, corollary 1.103 and \cite{Varadarajan84}, section 3.6.

Theorems~ \ref{tangential space} and~\ref{expnodiffeomorphism} may seem contradictory at first, as they state, that 
$$\textrm{Lie group} \qquad\longrightarrow\qquad\textrm{Lie algebra}:\quad \textrm{Good behavior!} $$
$$\textrm{Lie algebra} \qquad\longrightarrow\qquad\textrm{Lie group}:\quad  \textrm{Bad behavior!} $$
Nevertheless, it is a typical for infinite dimensional systems in the following sense:

The step from the Lie group to the Lie algebra is linearization or differentiation; differentiation exists in very general frameworks; typically no subtle obstacles arise; in contrast the step from the Lie algebra to the Lie group is integration; integration is in general frameworks a very subtle procedure~\cite{Kriegl97, Bertram04, Bertram08, Grothendieck97}.
This behavior seems contradictory as the usual intuition tells us there should be a geodesic in every direction. 
From the point of view of classical analysis, one can describe the analytic basics of this situation as a consequence of a change of limits, arising because of the non-compactness of $\mathbb C^*$:

\begin{itemize}
\item On the one hand, theorem~\ref{tangential space} describes locally uniform convergence on each compact subset $K\subset \mathbb C^*$.
\item Theorem~\ref{expnodiffeomorphism} shows in contrast that globally on $\mathbb{C}^*$ convergence needs no longer be uniform.
\end{itemize}

\noindent We now give the proofs that we have omitted:

\begin{proof}[Proof of theorem~\ref{tangential space}]~
\begin{enumerate}
\item We prove first the inclusion $M\mathfrak{g}\subset T_e(MG)$. Hence we have to show the following: let $\gamma(t): (-\epsilon, \epsilon) \rightarrow MG$ be a curve in $MG$ such that $\gamma(0)= e\in MG$. Then $\dot{\gamma}(0)$ is in $M\mathfrak g$. Using, that $M\mathfrak{g}$ is a loop algebra, we have an equivalent description of $\dot{\gamma}(0)$ as a function $\dot{\gamma}(0,z: \mathbb{C}^*)\longrightarrow \mathfrak{g}$, describing explicitly the loop.

Let $z_0\in \mathbb{C}^*$ arbitrary but fixed. Then $\gamma_{z_0}(t)$, the evaluation of $\gamma(t)$ at $z_0\in \mathbb{C^*}$, is a curve in $G$ satisfying $\gamma_{z_0}(0)=e$. Thus, using that the finite dimensional exponential map $\exp: \mathfrak{g}\longrightarrow G$ is a local diffeomorphism, we get that $\frac{d}{dt}\gamma_{z_0}(t)$ is an element in $\mathfrak{g}$. As $z_0$ is arbitrary we have constructed a function $w: \mathbb{C}^*\longrightarrow \mathfrak{g}$, satisfying $$w(z)=\dot{\gamma}(0,z)\,;$$ we have now to prove that $w$ is holomorphic.

We use, that $\gamma(t,z)$ depends holomorphically on $z$. This is equivalent to $$\frac{d}{d\overline{z}}\gamma(t,z)=0\,.$$
Hence $$0=\frac{d}{dt}\frac{d}{d\overline{z}}\gamma(t,z)=\frac{d}{d\overline{z}}\left[\frac{d}{dt}\gamma(t,z)\right]\,,$$
where we may switch the derivatives because of smoothness. Hence $\frac{d}{dt}\gamma(t,z)$ is holomorphic.

This shows: $T_e(MG)\subset M\mathfrak{g}$.

\item The other direction is straight forward: let $\gamma\in M\mathfrak{g}$. Then $\Mexp(t\gamma)$ is a curve in $MG$ with tangential vector $\gamma$. So $ M\mathfrak{g}\subset T_e(MG)$.

\end{enumerate}
This completes the proof. 
\end{proof}

\begin{proof}[Proof of theorem~\ref{expnodiffeomorphism}] 
The proof of theorem~\ref{expnodiffeomorphism} relies on the fact that $\Mexp$ is no local diffeomorphism for $SL(2,\mathbb{C})$. 
Let $\mathfrak{h}\simeq \mathfrak{sl}(2,\mathbb{C})\subset \mathfrak{g}$ be a subalgebra of $\mathfrak{g}$. Then there is an $H=SL(2,\mathbb{C})$ subgroup in $G$, such that $\mathfrak{h}$ is the Lie algebra of  $H$. See part VII.5 of the book~\cite{Knapp96}. 
Study the subalgebra $M\mathfrak{h}\subset M\mathfrak g$ and the subgroup $MH\subset MG$. $M\mathfrak{h}$ can be identified with $T_e(MH)$. Moreover $\Mexp: M\mathfrak{h} \subset MH$. But as example~\ref{sl2chasnodiffeomorphicexponentialmap} shows, $\Mexp$ is no local diffeomorphism. This completes the proof. 
\end{proof}

The image of the exponential map consists of those loops whose image is completely
contained in the image of the Lie group exponential map. We give further
comments about this situation in the third part, dealing with geometric aspects~\cite{Freyn12e}, where we study the
behavior of geodesics.

\begin{proof}[Proof of theorem~\ref{diffeomorphicexponentialmap}]
Let $\widetilde G$ be the universal cover of $G$; the exponential map $\exp: \mathfrak{g}\longrightarrow \widetilde{G}$ is a biholomorphic map. Thus cocatenation with $\exp$ (resp.\ $\exp^{-1}$) induces a biholomorphic map between $M\mathfrak{g}$ and $M\widetilde{G}$. To get that $\Mexp: M\mathfrak{g}\longrightarrow MG$ is a local diffeomorphism, one uses the fact that each loop in $\widetilde{MG}$ projects onto a loop in $MG$ and, conversely, each loop in $MG$ can be lifted to a loop in $M\widetilde{G}$, which is unique up to Deck transformation and thus locally unique.  This proves that $\Mexp$ is a local diffeomorphism. 
\end{proof}

\subsubsection{Manifold structures on groups of holomorphic maps}
\label{Manifoldstructuresongroupsofholomorphicmaps}

In this section we prove that the groups $MG_{\mathbb{R}}$, $MG_{\mathbb{C}}$, and various quotients are tame Fr\'echet manifolds.

\begin{theorem}
\label{mgcistame}
$MG_{\mathbb C}$ is a tame Fr\'echet manifold.
\end{theorem}

The idea of the proof is to use logarithmic derivatives. The idea to use logarithmic derivatives to get charts is quite common: for regular Lie groups it is developed in the book~\cite{Kriegl97}, chapters 38 and 40. 
Furthermore it is used by Karl-Hermann Neeb~\cite{Neeb06} to prove the following theorem~\cite{Neeb06}, theorem III.1.9.:

\begin{theorem}[Neeb]
\label{Neebstheorem}
Let $\mathbb{F}\in \{\mathbb{R}, \mathbb{C}\}$, $G$ be a $\mathbb{F}$-Lie group and $M$ a finite dimensional, connected $\sigma$-compact
$\mathbb{F}$-manifold. We endow the group $C_{\mathbb{F}}^{\infty}(M,G)$ with the compact open $C^{\infty}$-topology, turning it into a topological group. This topology is compatible with a Lie group structure if
$\dim_{\mathbb{F}}M=1$ and $\pi_1(M)$ is finitely generated.
\end{theorem}

The main ingredients for the proof of Neebs theorem are the use of logarithmic derivatives to define charts in $\Omega^1(M, \mathfrak{g})$, the space of $\mathfrak{g}$-valued $1$-forms on $M$ and Gl\"ockners inverse function theorem~\cite{Glockner07}, to take care of the monodromy if $\pi_1(M)$ is nontrivial.

If $\mathbb{F}=\mathbb{C}$ we have the equivalence $C_{\mathbb{F}}^{\infty}(M,G)\simeq Hol(M,G)$. Our situation is the special case $M=\mathbb{C}^*$. Hence, $\pi_1(M)=\mathbb{Z}$ is a finitely generated group.  The compact open $C^{\infty}$-topology coincides for holomorphic maps with the topology of compact convergence. 
Thus Neeb's theorem tells us that  $MG_{\mathbb F}, \mathbb{F} \in \{\mathbb{R}, \mathbb{C}\}$ are locally convex
topological Lie groups.

Nevertheless, we do not get tame structures. Hence we have to prove the theorem completely new. Our presentation follows the proof of Karl-Hermann Neeb for the locally convex case.

We need some definitions: $\alpha \in \Omega^1(M,\mathfrak{g})$ is called integrable iff there exists a function $f\in Hol(M, G)$ such that $\delta(f):=f^{-1}df=\alpha$. The uniqueness of solutions to linear differential equations shows that $\delta(f_1)=\delta(f_2)$ iff $f_1=gf_2$ for some $g\in G$.

The first step of the proof is the following lemma, whose statement and proof can be found in~\cite{Kriegl97} and \cite{Neeb06}:
morally it is a straight forward application of the monodromy principle for holomorphic Pfaffian systems, as described in the article~\cite{Novikov02}.

\begin{lemma}
Let $M$ be a $1$-dimensional complex manifold, $\alpha \in \Omega^{1}(M, \mathfrak{g})$.
\label{integrability of 1 dim }
\begin{enumerate}
	\item $\alpha$ is locally integrable
	\item If $M$ is connected, $M_0\in M$,  then there exists a homomorphism 
	$$ \textrm{per}_{\alpha}: \pi_1(M,m_o)\longrightarrow G$$
        that vanishes iff $\alpha$ is integrable.
\end{enumerate} 
\end{lemma}

\begin{proof} Proof of theorem \ref{mgcistame}
\noindent  We define the embedding:
\[
\begin{array}{ccr}
\varphi: MG_{\mathbb{C}}&\hookrightarrow& \Omega^{1}(\mathbb{C}^*, \mathfrak{g}_{\mathbb{C}}) \times G_{\mathbb{C}}\, ,\\
f &\mapsto& (\delta(f)=f^{-1}df, f(1))\, .
\end{array}
\]
This embedding is injective as $(\delta(f_1), f_1(1))=(\delta(f_2), f_2(1))$ iff $\delta(f_1)=\delta(f_2)$ and $f_1(1)=f_2(1)$. Using the uniqueness of solutions (up to the starting point) of the linear differential equation $df=f\omega$ where $\omega:=f_1^{-1}df_{1}$, the first condition leads to the relation $f_1=g f_2$ for some $g\in G$. Then the second condition gives uniqueness as $gf_{2}(1)=f_{2}(1)$ leads to $g=\textrm{Id}$. 

Compare this embedding with the description of polar actions on Fr\'echet spaces in section~\ref{PolaractionsontameFrechetspaces}.
Let $\pi_1$ and $\pi_2$ denote the projections:
\[
\begin{array}{ccc}
\pi_1: \Omega^{1}(\mathbb{C}^*, \mathfrak{g}_{\mathbb{C}}) \times G_{\mathbb{C}} &\mapsto& \Omega^{1}(\mathbb{C}^*, \mathfrak{g}_{\mathbb{C}})\, ,\\
\pi_2: \Omega^{1}(\mathbb{C}^*, \mathfrak{g}_{\mathbb{C}}) \times G_{\mathbb{C}} &\mapsto& G_{\mathbb{C}}\, .
\end{array}
\]

We construct charts for $MG_{\mathbb{C}}\subset \Omega^{1}(\mathbb{C}^*, \mathfrak{g}_{\mathbb{C}}) \times G_{\mathbb{C}}$ as direct product of charts for $\pi_1 \circ \varphi(MG_{\mathbb{C}})$ and $\pi_2 \circ \varphi(MG_{\mathbb{C}})$.
\begin{itemize}
  \item $\pi_2 \circ \varphi$ is surjective; hence a describing of charts for the second factor can be done by choosing charts for $G$. Via the exponential mapping and     left translation, we get charts $\psi_{2,g}: U(g)\longrightarrow V(0)$ defined on an open set $U(g)$ around $g\in G$ taking values in $V(0)\subset          \mathfrak{g}_{\mathbb{C}}$. To describe the family of norms, we use the Euclidean metric $$\|\hspace{3pt}\|_n:=\|\hspace{3pt}\|_{Eucl}\, .$$ 

  \item The first factor is more difficult to deal with as $\pi_1 \circ \varphi$ is not surjective. While every $\mathfrak{g}_{\mathbb{C}}$-valued $1$-form      $\alpha \in \Omega(\mathbb{C}^*, \mathfrak{g}_{\mathbb{C}})$ is locally integrable by lemma \ref{integrability of 1 dim }, the monodromy may prevent global integrability. 
   A form $\alpha \in \Omega^1(\mathbb{C}^*, \mathfrak{g}_{\mathbb{C}})$ is in the image of $\pi_1 \circ \varphi$ iff its monodromy vanishes, that is iff
   $$e^{\int_{S^1}\alpha}=e\in G_{\mathbb{C}}\,.$$

    This is equivalent to the condition $\int_{S^1}\alpha = a_{-1}(\alpha) \subset \frac{1}{2\pi i} \exp^{-1}(e)$ where $a_{-1}(\alpha) $ denotes the      $(-1)$-Laurent coefficient of the Laurent series of $\alpha= f(z)dz$. Hence we get the characterization of  $\Im(\pi_1\circ \varphi)$ as the inverse image of $e\in G_{\mathbb{C}}$ via the monodromy map.

Thus we have to show that this inverse image is a tame Fr\'echet manifold. To this end, we use composition with a chart $\psi: U\longrightarrow V$ for $e\in U\subset G$ with values in $G_{\mathbb{C}}$. This gives us a tame map $\Omega(\mathbb{C}^*, \mathfrak{g}_{\mathbb{C}})\longrightarrow \mathfrak{g}_{\mathbb{C}}$. This map satisfies the assumptions of the following theorem, whose proof can be found in~\cite{Freyn12c}.

\begin{theorem}
\label{submanifoldsoffrechetspace}
Let $F$ be a tame space and $\varphi: F \rightarrow \mathbb{R}^n$ a tame map. Let $g\in \mathbb{R}^n$ be a regular value for $\varphi$.
Then $\varphi^{-1}(g)$ is a tame Fr\'echet submanifold of finite type. 
\end{theorem}

 Thus its inverse image is a tame Fr\'echet submanifold. This proves that $\pi_1 \circ \varphi$ is a tame Fr\'echet submanifold.
\end{itemize}

\noindent Thus $MG_{\mathbb{C}}$ as a product of a tame Fr\'echet manifold with a Lie group is a tame Fr\'echet manifold.
This completes the proof of theorem~\ref{mgcistame}.
\end{proof}

\begin{theorem}
$MG_{\mathbb{C}}$ is a tame Fr\'echet Lie group.
\end{theorem}

\begin{proof}
From theorem \ref{mgcistame} we know that $MG_{\mathbb{C}}$ is tame  Fr\'echet manifold. Hence we have to check that 
\begin{displaymath}
\psi_1: MG_{\mathbb{C}}\times MG_{\mathbb{C}}\longrightarrow MG_{\mathbb{C}},\quad (g,f)\mapsto gf,\end{displaymath}
and 
\begin{displaymath}
\psi_2:MG_{\mathbb{C}}\longrightarrow MG_{\mathbb{C}},\quad f\mapsto f^{-1}
\end{displaymath}
are tame Fr\'echet maps.
Using the embedding
\[
\begin{array}{ccr}
\varphi: MG_{\mathbb{C}}&\hookrightarrow& \Omega^{1}(\mathbb{C}^*, \mathfrak{g}_{\mathbb{C}}) \times G_{\mathbb{C}}\, ,\\
f &\mapsto& (f^{-1}df, f(1))\, .
\end{array}
\]
we get for $\psi_1$ the description
\begin{align*}
\psi_1: \Omega^{1}(\mathbb{C}^*, \mathfrak{g}_{\mathbb{C}})\times\Omega^{1}(\mathbb{C}^*, \mathfrak{g}_{\mathbb{C}}) \times G_{\mathbb{C}}\times G_{\mathbb{C}}&\longrightarrow \Omega^{1}(\mathbb{C}^*, \mathfrak{g}_{\mathbb{C}}) \times G_{\mathbb{C}}\, \\
(f^{-1}df, g^{-1}dg, f(1), g(1)) &\mapsto (gf^{-1}dfg^{-1}+g^{-1}dg, f(1)g(1))\, .
\end{align*}
which is smooth as a direct product of smooth tame maps. 

Similarly we get for $\psi_2$:
\begin{align*}
\psi_1: \Omega^{1}(\mathbb{C}^*, \mathfrak{g}_{\mathbb{C}})\times G_{\mathbb{C}}&\longrightarrow \Omega^{1}(\mathbb{C}^*, \mathfrak{g}_{\mathbb{C}}) \times G_{\mathbb{C}}\, \\
(f^{-1}df,  f(1))  &\mapsto (f(f^{-1}df) f^{-1}, f^{-1}(1))\, .
\end{align*}
which is again a smooth tame map.

\end{proof}

We now investigate different classes of quotients of loop groups:

\begin{proposition}
$MG_{\mathbb{R}}$ is a tame Fr\'echet Lie group.
\end{proposition}

\begin{proof}
The proof is similar to the proof for $MG_{\mathbb{C}}$. We have only to take care of the reality condition $f(S^1)\subset G_{\mathbb{R}}$ for loops $f\in MG_{\mathbb{C}}$.
Thus the embedding $\varphi$ maps a loop $f$ into $\Omega \left(\mathbb{C}^*, \mathfrak{g}_{\mathbb{R}}\right)\times G_{\mathbb{R}}$, which are both tame Fr\'echet spaces. Now similar arguments apply.
\end{proof}

\begin{proposition}
\label{mgc/mgrtamefrechet}
The group $MG_{\mathbb{C}}/MG_{\mathbb R}$ is a tame Fr\'echet manifold.
\end{proposition}
\begin{proof}

Review the embedding 
$$\psi: MG_{\mathbb{C}}\longrightarrow \Omega^1(\mathbb{C^*}, \mathfrak{g}_{\mathbb{C}})\times G_{\mathbb{C}}\,.$$

A loop $f\cdot g$ is mapped onto $\psi(f\cdot g)=(g^{-1}f^{-1}df g+ g^{-1}dg, [f\cdot g](1))$. This is the well-known gauge-action of the group $MG_{\mathbb{R}}$, denoted by $\mathcal{G}^*(MG_{\mathbb{R}})$.
Thus there is a well defined embedding
$$\psi: MG_{\mathbb{C}}/MG_{\mathbb{R}}\longrightarrow \Omega^1(\mathbb{C^*}, \mathfrak{g}_{\mathbb{C}})/\mathcal{G}^*(MG_{\mathbb{R}})\times G_{\mathbb{C}}/G_{\mathbb{R}}\,.$$

No we study again the projections $\pi_1$ and $\pi_2$ on the first and second factor. $\pi_2$ is surjective; $G_{\mathbb{C}}/G_{\mathbb{R}}$ is a tame manifold; so this factor is no problem.

The projection on the first factor, $\pi_1$, needs a more careful analysis:
the right multiplication of $MG_{\mathbb{R}}$ on $MG_{\mathbb{R}}$ is surjective: using the decomposition $\mathfrak{g}_{\mathbb{C}}=\mathfrak{g}_{\mathbb{R}}+i \mathfrak{g}_{\mathbb{R}}$, we get for $\Omega^1(\mathbb{C}^*, \mathfrak{g}_{\mathbb{C}}):=\Omega^1_{\mathbb{R}}(\mathbb{C}^*, \mathfrak{g}_{\mathbb{C}})+i\Omega^1_{\mathbb{R}} (\mathbb{C}^*, \mathfrak{g}_{\mathbb{C}})$.  

The surjectivity of the right multiplication of $MG_{\mathbb{R}}$ on $MG_{\mathbb{R}}$ translates into the surjectivity of the $MG$-gauge action on the imaginary part $i\Omega^1_{\mathbb{R}} (\mathbb{C}^*, \mathfrak{g}_{\mathbb{C}})\cap \Im(\pi_1 \circ \psi)(MG)$. Thus we can suppose to have chosen a representative $f\in f\cdot MG$, such that the imaginary part $ \pi_1\circ \psi(f)$ is $0$.
So all we have to check is the real part. Here we find that $\exp^{-1}(e)=0$. Thus $a_{-1}=0$. So we can identify the image $\pi_1 \circ \psi (MG_{\mathbb{C}}/MG_{\mathbb{R}})\simeq \Omega^1_{\mathbb{R}}(\mathbb{C^*}, \mathfrak{g}|a_{-1}=0) $. This is a tame Fr\'echet space.

Hence proposition~\ref{mgc/mgrtamefrechet} is proven.
\end{proof}

Having proved that $MG_{\mathbb{C}}$ and $MG_{\mathbb{C}}/MG_{\mathbb{R}}$ are tame Fr\'echet
manifolds we have to check that the same is true for the quotients
$MG_{\mathbb R}/\textrm{Fix}{(\rho)}$ and $MG_D/\textrm{Fix}{(\rho)}$.

To this end, let $MG$ be a loop group and $M\rho$ the loop part of an involution of the second kind.

\begin{proposition}
Let $MG_D$ be a non-compact real form of $MG_{\mathbb{C}}$. $MG_D$ is a tame Fr\'echet manifold. 
\end{proposition}

\begin{proposition}
The quotient spaces $MG_{\mathbb{R}}/\textrm{Fix}(M\rho)$ and $MG_{D, \rho}/\textrm{Fix}(M\rho)$ are tame manifolds.
\end{proposition}

\begin{proof}
The proof is  an argument analogous to the proof that $MG_{\mathbb{C}}/MG_{\mathbb{R}}$ is a tame Fr\'echet space.
\end{proof}

\noindent The next class are twisted loop groups:

\begin{proposition}[Twisted loop groups]
Let $G$ be a compact simple Lie group of type $A_n, D_n$ or  $E_6$ and $\sigma$ a diagram automorphism of order $m\in \{2,3\}$. Let $\omega=e^{\frac{2\pi i}{m}}$.

\begin{itemize}
\item The group $A_nG^{\sigma}:=\{f\in A^nG | \sigma \circ f(z)=f(\omega z)\}$ is a Banach-Lie group. 
\item The group $MG^{\sigma}:=\{f\in MG | \sigma \circ f(z)=f(\omega z)\}$ is a tame manifold. Charts can be taken to be in $\Omega^1(\mathbb{C}^*, G)^{\sigma}$. Furthermore $M\mathfrak{g}^{\sigma}\simeq T_e(MG^{\sigma})$. 
\end{itemize}
\end{proposition}

\begin{proof}
To generalize the proofs of the non-twisted setting to the twisted setting one has to check that the subspaces defined by diagram automorphism are preserved by the logarithmic derivative.

\begin{enumerate}
\item For the exponential map,  we use $\sigma_{\mathfrak{g}}$ resp.\ $\sigma_{G}$ to denote the realization of the diagram automorphism $\sigma$ on $\mathfrak{g}$ resp.\ $G$. Any involution of a semisimple Lie group satisfies the identity: $\sigma_{G} \circ \exp =\exp \circ {\sigma_{\mathfrak{g}}}$
$$[\sigma_{G}\circ \Mexp(f)](z)=\exp(\sigma_{\mathfrak{g}}(f(z)))=\exp(f(\omega z))=[\Mexp(f)](\omega z),$$

\item For the logarithmic derivative we calculate:
$$\delta(\sigma \circ f)=(\sigma f)^{-1} d(\sigma \circ f)= \sigma f^{-1} \sigma df = \sigma (\delta f)\,.$$
Thus we get charts in the $\sigma$-invariant subalgebra of $\Omega^1(\mathbb{C}^*, \mathfrak{g}_{\mathbb{C}})$.
\end{enumerate}
\end{proof}

\noindent The following definition is due to Omori~\cite{Omori97}:

\begin{definition}[Exponential pair]
\label{exponentialtype}
A pair $(G,\mathfrak{g})$ consisting of a Fr\'echet group $G$ and a Fr\'echet space $\mathfrak{g}$ is called a topological group of \emph{exponential type} if there is a continuous mapping:
$$\exp: \mathfrak{g}\longrightarrow G$$
such that:
\begin{enumerate}
	\item For every $X \in \mathfrak{g}$, $\exp(sX)$ is a one-parameter subgroup of $G$.
	\item For $X,Y \in \mathfrak{g}$, $X=Y$ iff $\exp(sX)=\exp(sY)$ for every $s\in \mathbb{R}$.
	\item For a sequence $\{X_n\}\in \mathfrak{g}$, $\displaystyle\lim_{n\rightarrow \infty} X_n$ converges to an element $X\in \mathfrak{g}$ iff $\displaystyle\lim_{n\rightarrow \infty} (\exp{sX_n})$ converges uniformly on each compact interval to the element $\exp(sX)$.
	\item There is a continuous mapping $\textrm{Ad}: G\times \mathfrak{g}\longrightarrow \mathfrak{g} $ with  $h\exp(sX)h^{-1}=\exp s \textrm{Ad}(h)X$ for every $h\in G$ and $X\in \mathfrak{g}$.
\end{enumerate}
\end{definition}

\begin{proposition}[Exponential type]
The pair $(MG_{\mathbb{K}}, M\mathfrak{g}_{\mathbb{K}})$ is of exponential type.
\end{proposition}

\begin{proof}
We proved that $MG$ is a tame Fr\'echet Lie group, thus a topological group.
To prove that $(MG_{\mathbb{K}}, M\mathfrak{g}_{\mathbb{K}})$ is of exponential type, we have to check the four conditions given in definition~\ref{exponentialtype}:

\begin{enumerate}
	\item The first condition can be checked by a pointwise analysis: for $f\in M\mathfrak{g}$ and for every $z\in \mathbb{C}^*$, the curve $\exp{sf(z)}$ is a $1$-parameter subgroup in $G$. This pieces together for all $z\in \mathbb{C^*}$, to yield the condition.
	\item The second condition follows analogously: let $X,Y\in M\mathfrak{g}$. $X=Y$ iff $X(z)=Y(z)$ for all $z \in \mathbb{C}^*$. The finite dimensional theory tells us that this is equivalent to the curves $\exp(sX(z))\subset G_{\mathbb{K}}$ and $\exp(sY(z))\subset G_{\mathbb{K}}$ to be equivalent for all $z\in \mathbb{C}^*$, but this is equivalent to $\exp(sX)=\exp(sY)$.
	\item Let $\{X_n(z)\}\in M\mathfrak{g}, z\in \mathbb{C}$ be a sequence of elements such that  $\displaystyle\lim_{n\rightarrow \infty} X_n(z)=X \in M\mathfrak{g}$. Let $T\subset \mathbb{R}$ be a compact interval, $s\in T$. 
	As we have on $MG$ the compact-open topology, 
	$$\lim_{n\rightarrow \infty}(\exp{sX_n})=\exp(sX)\Leftrightarrow \forall K \subset \mathbb{C^*}:\exp\lim_{n\rightarrow \infty}(\exp{sX_n(K)})=\exp(sX)(K)\,.$$
	This assertion is correct as for every $z\in K:\exp\displaystyle\lim_{n\rightarrow \infty}(\exp{sX_n(z)})=\exp(sX)(z)$.
	\item The last assertion follows again from pointwise consideration and the validity of the assertion for finite semisimple Lie groups.\qedhere
\end{enumerate}
\end{proof}

We give some remarks about $1$-parameter subgroups. 

\begin{remark}
Let $g(t):=X\exp(tu)$ for $u\in X\mathfrak{g}_{\mathbb{K}}$ be a $1$-parameter subgroup in $XG_{\mathbb{K}}$, $X\in \{A_n , \mathbb{C}^*\}$, $\mathbb{K}\in \{\mathbb{R}, \mathbb{C}\}$. Then the following statements hold:
\begin{enumerate}
\item  $X\exp(tu)_{z_0}$ is a $1$-parameter group in $G_{ \mathbb{C}}$ for all $z_0\in X$.
\item If $A_n\subset A_{n+k}$ then the embedding $A_{n+k}G \hookrightarrow A_nG$ maps $1$-parameter subgroups onto $1$-parameter subgroups.
\end{enumerate}
\end{remark}

\begin{proof} Direct calculation. \end{proof}

\begin{remark}
As we have seen, the fact that the exponential function does not define a local diffeomorphism is responsible for several difficulties; so it is reasonable to try to use a setting in which the exponential function defines a local diffeomorphism. 
So let us try to take loops $f: S^1\longrightarrow G$ satisfying some regularity condition. In this case the exponential map defines always a local diffeomorphism, as a neighborhood of
the identity element of such a loop group is given by loops whose images lie in
a small neighborhood $V$ of the identity of the subjacent Lie group; this
neighborhood can be chosen in a way such that the group exponential is a
diffeomorphism from an open neighborhood $U$ in the Lie algebra onto it. But now other problems appear:
\begin{enumerate}
\item Suppose the functions to be $H^1$-Sobolev loops. In this setting, one can construct weak Hilbert symmetric spaces of compact and non-compact type. Nevertheless, one cannot define the double extension corresponding to the $c$- and $d$-part of the Kac-Moody algebra. As this extension is responsible for the structure theory, this setting is not useful for us. 
\item To be able to construct the extension corresponding to the derivative $d$, one needs loops that are $C^{\infty}$. For $C^{\infty}$-loops, it is possible to construct compact type symmetric spaces corresponding to the finite dimensional types $I$ and $III$, but for Kac-Moody symmetric spaces of $C^{\infty}$-loops there is no dualization: as the complexification of $c$ is not defined, we cannot complexify. As a consequence there are no symmetric spaces of the non-compact type. 
\end{enumerate}
The details of both theories are developed in~\cite{Popescu05}. Summarizing we have the following observation:

\begin{proposition}
The setting of holomorphic loops is the biggest setting such that Kac-Moody symmetric space of the compact type and of the non-compact type of the same regularity condition can be defined.
\end{proposition}

Let us mention in this context the short summary in \cite{Berger03} about infinite dimensional differential geometry, where the conflict between often desirable structure as a Hilbert space and useful metrics is addressed.
\end{remark}

\section{Polar actions on tame Fr\'echet spaces}
\label{PolaractionsontameFrechetspaces}

In this section we study polar representations on tame Fr\'echet vector space. The aim is, to show, that the adjoint representation of Kac-Moody groups of holomorphic maps $\widehat{MG}$ and their associated $s$-representations induce polar actions; from the point of view of Kac-Moody geometry, this step assures, that there is a well-defined connection between the (local) classes of objects as are polar actions and isoparametric submanifolds and the global classes of objects i.e. affine Kac-Moody symmetric spaces and twin cities.
 
Let us start by sketching the finite dimensional blueprint: Let $M=G/K$ be a finite dimensional Riemannian symmetric space. The group $K$ stabilizes the point $p=eK\in M$; it is called the isotropy subgroup. The linear representation $K\equiv K_p:T_pM\longrightarrow T_pM$ is called the isotropy representation of $M$. Recall furthermore that a representation $K:V\longrightarrow V$ of a Lie group $K$ on a vector space $V$ is polar, if there is a subspace $\Sigma\subset V$, called a section, that intersects each orbit orthogonally. It is a by now classical result, that the isotropy representation of a (finite dimensional) Riemann symmetric space is a \emph{polar representation} of the isotropy group on the tangential space. For the isotropy representation of a symmetric space, sections correspond to any maximal flat subspace, containing $p$. We prove in a subsequent part of our program describing the geometry of Kac-Moody symmetric spaces, dealing with the geometric properties of these spaces, that Kac-Moody symmetric spaces behave in a similar way. Remark nevertheless a striking difference between the finite dimensional (spherical) and the infinite dimensional affine situation: while in the finite dimensional situation all orbits have finite codimenions, ranging from $2$ to the rank, in the affine case there are orbits with (small) finite codimension bounded by $rk(G)$ and orbits with infinite codimension. This phenomenon is related to the dichothomy of complex affine Kac-Moody groups, that can be viewed either as (finite dimensional) (algebraic) groups over the ring of holomorphic functions or as (infinite dimensional) affine Kac-Moody group over the base field $\mathbb{C}$ (resp. $\mathbb{R}$. Via duality constructions this dichothomy is also the algebraic structure behind the differences in the geometry of coadjoint orbits of Kac-Moody groups and loop groups, which is sketched for example in the book~\cite{Khesin09}; we gave a building theoretic interpretation in~\cite{Freyn10d}, relating the finite dimensional orbits to the affine twin city and the infinite dimensional ones to the spherical building at infinity~\cite{Freyn11a}. 

We prove in this section, that the orbits with finite codimension correspond to the gauge actions of tame loop groups on tame spaces; these actions are restrictions of the adjoint action of the Kac-Moody group to certain subspaces. 

Closely related is the theory of polar actions on Hilbert spaces, which is described in the article~\cite{HPTT, PinkallThorbergsson90, Terng95, Gross00}.

The fundamental theorem due to C.-L.\ Terng states:
\begin{theorem}
Define $P(G,H):=\{g\in H^1([0,1], G)|(g(0), g(1))\in H\subset G\times G\}$ and $V=H^0([0,1],\mathfrak{g})$. Suppose the $H$-action on $G$ is polar with flat sections. Let $A$ be a torus section through $e$ and let $\mathfrak{a}$ denote its Lie algebra. Then the gauge action of $P(G,H)$ on $V$ is polar with section $\mathfrak{a}$.
\end{theorem}

\begin{proof}see~\cite{Terng95}. \end{proof}

Important special cases are the following: let $\Delta_{\sigma}\subset G\times G$ denote the $\sigma$-twisted diagonal subgroup of $G\times G$, that is: $(g,h)\in \Delta_{\sigma}$ iff $h=\sigma(g)$. We use the notation $\Delta=\Delta_{\textrm{Id}}$ for the non-twisted subgroup.

\begin{enumerate}
\item The gauge action of $H^1$-Sobolev loop groups $P(G, G\times G)\cong H^1([0,1],G)$ on their $H^0$-Sobolev loop algebras $H^0([0,1],\mathfrak{g})$ is transitive.
\item The gauge action of $P(G, \Delta_{\sigma})$ on $H^1([0,1],G)$  is polar with flat sections \cite{HPTT}.
\item The gauge action of $P(G, K\times K)$ on $H^1([0,1],G)$ where $K$ is the fixed point set of some involution of $G$ is polar with flat sections \cite{HPTT}.
\item The gauge action of $P(G, K_1\times K_2)$ on $H^1([0,1],G)$ where $K_i, i\in \{1,2\}$ are the fixed point groups of involutions of $G$ is polar with flat sections \cite{HPTT}.
\item The gauge action of Sobolev-$H^1$-loop groups $H^1(S^1,G)$ on their Sobolev $H^0$-loop algebras $H^0(S^1,\mathfrak{g})$-is polar \cite{PalaisTerng88}. 
\end{enumerate}

In this section we describe a similar theory for the loop groups $XG^{\sigma}$ on the tame  loop algebras $X\mathfrak{g}^{\sigma}$. As usual let $X\in \{A_n, \mathbb{C}^*\}$. Holomorphic functions on $A_n$ are supposed to be holomorphic in an open set containing $A_n$. From the embedding $XG^{\sigma}\hookrightarrow H^1([0,1],G)$ and $X\mathfrak{g}^{\sigma}\hookrightarrow H^0([0,1],\mathfrak{g})$ it is clear that the algebraic part of the theory works exactly the same in all regularity conditions. This means for example: sections for holomorphic actions  correspond to sections for the Hilbert actions and the associated affine Weyl groups are the same. Hence the crucial point is to check that the additional regularity restrictions fit together. This breaks down to two points:

\begin{enumerate}
\item One has to show locally that the additional regularity conditions are satisfied. 
\item One has to show globally that monodromy conditions are satisfied. 
\end{enumerate}

To define orthogonality on $X\mathfrak{g}$ we use the $H^0$-scalar product induced on $M\mathfrak{g}$ by the embedding into $H^0([0,1],\mathfrak{g})$. Hence we can define polar actions on Banach (resp.\ tame) spaces like that:

\begin{definition}
An action of a Lie group $G$ on a Fr\'echet space $F$ is called polar iff there is a subspace $S$, called a section, intersecting each orbit orthogonally with respect to some scalar product. 
\end{definition}

\begin{theorem}
\label{polaractiononmg}
The gauge action of $XG_{\mathbb{R}}^{\sigma}$ on $X\mathfrak{g}^{\sigma}$ is polar; an Abelian subalgebra $\mathfrak{a} \subset \mathfrak{g}$ interpreted as constant loops is a section.
\end{theorem}

The proof consists of two parts:
\begin{enumerate}
\item We have to show that each orbit intersects the section $\mathfrak{a}$.  
\item We have to show that the intersection is orthogonal.
\end{enumerate}

\noindent The second part follows trivially from the embedding and Terng's result.

\noindent Thus we are left with proving the first assertion. We do this in a step-by-step way: first we study the action of $C^{k}$-loop groups on $C^{k-1}$-loop algebras ($k\in \{\mathbb{N}, \infty\}$). Then we proceed to the holomorphic setting of theorem~\ref{polaractiononmg}.

\begin{lemma}
\label{cinfty}
The gauge action of $L^{k}G^{\sigma}$ on $L^{k-1}\mathfrak{g}^{\sigma}$ is polar for $k\in \{\mathbb{N, \infty}\}$.
\end{lemma} 

\noindent This result is used without proof in~\cite{Popescu05} in order to show that all finite dimensional flats are conjugate. We do not know if a proof can be found in the literature. For completeness we give one: 

\begin{proof}[Proof of lemma~\ref{cinfty}]~
\begin{enumerate}
\item  {\bf Orthogonality} in $L^{k-1}\mathfrak{g}^{\sigma}$ is defined via the embedding into the space $H^0([0,1],\mathfrak{g})$ and the use of the $H^0$-scalar product. Hence orthogonality of the intersection between sections and orbits is covered by Terng's result.

\item {\bf Local regularity}~

\noindent Define the following spaces  

\begin{displaymath}
P(G,H)^k:=\{g\in C^k([0,1], G)|(g(0), g(1))\in H\subset G\times G\}\, .
\end{displaymath}

 Furthermore we use the equivalence~\cite{Terng95}
\begin{align*}
P(G;e\times G)&\simeq H^0([0,1],\mathfrak{g})\, ,\\
h&\leftrightarrow -h'h^{-1}\, .
\end{align*}

Terng's polarity result~\cite{Terng95} yields that the action of $P(G,\Delta_{\sigma})$ on $P(G; e\times G)$ defined by $(g(t), h(t))\mapsto g(t) h(t) g(0)^{-1}$ is polar with  a section of constant loops $\exp{t\mathfrak{a}}$ where $\mathfrak{a}$ is a maximal Abelian subalgebra in $\mathfrak{g}$ (if $\sigma\not=0$ we restrict to $\mathfrak{a}_{\sigma}$ and omit the $\sigma$ in the notation~\cite{Kac90}). Thus for every $h(t)\in P(G;e\times G)$ there exist $g(t)\in P(G,\Delta_{\sigma})$ and $X\in \mathfrak{a}$, such that $g(t)h(t)g(0)^{-1}=\exp(tX)$.  

Rearranging this equation we deduce for any  loop $ g(t) \in P(G, \Delta_{\sigma})$ the explicit description $g(t):= \exp(tX)g(0)h(t)^{-1}$. Hence if $h(t)\in P^k(G;e\times G)$ then $g(t)\in  P^k(G,\Delta_{\sigma})$. Combining this with the orthogonality we obtain that the actions of $P^k(G,\Delta_{\sigma})$ on $P^k(G;e\times G)\simeq H^{k-1}([0,1],\mathfrak{g})$ and of $P^{\infty}(G,\Delta_{\sigma})$ on $P^{\infty}(G;e\times G)\simeq H^{\infty}([0,1],\mathfrak{g})$ are polar.

\item {\bf The periodicity relation:}~
We want to show that $L^{k}G^{\sigma}$ acts on $L^{k-1}\mathfrak{g}^{\sigma}$ with slice $\mathfrak{a}$ for $k\in \{\mathbb{N}, \infty\}$.

  \begin{enumerate}
  \item Let first $g\in L^{k}G^{\sigma}$ and $u\in L^{k-1}\mathfrak{g}^{\sigma}$. Then $g \cdot u= gug^{-1}-g' g^{-1}$ is in $L^{k-1}\mathfrak{g}^{\sigma}$. Thus $L^{k}G^{\sigma}$ acts on $L^{k-1}\mathfrak{g}^{\sigma}$.
  \item We have to show that any $L^{k}G^{\sigma}$-orbit intersects the section $\mathfrak{a}$. This is equivalent to:
 For each $u\in L^{k-1}\mathfrak{g}^{\sigma}$, there is $X\in \mathfrak{g}$ and $g\in P^k(G,\Delta)$ such that
$\exp(tX)=g(t)h(t)g^{-1}(0)$ with $h'(t)=u(t)h(t)$ and the derivatives coincide; interpret in this last equation $u(t)$ as a quasi-periodic function on $\mathbb{R}$ (i.e. $u(t+2\pi)=\sigma u(t)$ and $h(t)$ as a function on $\mathbb{R}$). 

\noindent Using the first part, we find a function $g(t)\in P^k(G,\Delta_{\sigma})$. Hence, what remains is to check the closing condition of the derivatives: $g^{(n)}\cdot u(2\pi)= \sigma g^{(n)}\cdot u(0)$. We prove that it is equivalent to the closing condition $g^{(n+1)}(2\pi)=\sigma g^{(n+1)}(0)$.

We start with the case $n=1$.
For this case, we have to show
$$\exp((t+2\pi)X) g_0 h(t+2\pi)^{-1}= \sigma (\exp(tX)g_0 h(t)^{-1})\,. $$ 

After rearranging, this is equivalent to the identity 
$$\sigma(g_0^{-1})\exp(2\pi X) g_0=\sigma(h(t)^{-1})h(t+2\pi)\,.$$
As the left side is a constant we find:
$$\left(\sigma(h(t)^{-1})h(t+2\pi)\right)'=0\,.$$ 
Hence: $-\sigma(h(t)^{-1}h'(t) h(t)^{-1})h(t+2\pi)+ \sigma(h(t)^{-1})h'(t+2\pi)=0$.
Rearranging this equality we get
$$\sigma(u(t))= - \sigma(h'(t))\sigma( h(t)^{-1})= -h'(t+2\pi)h(t+2\pi)^{-1}=u(t+2\pi)\,$$
which is the desired periodicity condition. 

For $n\not=0$ we use induction. If $g$ is $k$-times differentiable then $u$ is $k-1$-times differentiable. 
This proves the lemma.\qedhere
   \end{enumerate}
\end{enumerate}
\end{proof}

\begin{proof}[Proof of theorem~\ref{polaractiononmg}]
To prove the theorem, we have to further strengthen the used regularity conditions to holomorphic functions. 
The description in the proof of lemma~\ref{cinfty} shows that the group of analytic loops $L_{an}(S^1, G)$ acts polarly with section $\mathfrak{a}$ on the algebra $L_{an}(S^1, \mathfrak{g})$ of analytic loops. 
\begin{enumerate}
\item {\bf The case of holomorphic loops on $\mathbb{C}^*$}~
For the specialization to holomorphic maps we use the description:
$$H_{\mathbb{C}^*}(\mathbb{C},\mathfrak{g}_{\mathbb{C}})_{\mathbb{R}}:=\{f:\mathbb{C}\longrightarrow \mathfrak{g}_{\mathbb{C}}|f(z+i\mathbb{Z})=f(z), f{i\mathbb{R}}\subset \mathfrak{g}_{\mathbb{R}}\}\,.$$
\noindent Identifying $it \leftrightarrow t$ in this (resp.\ the above) description we get an embedding: 
  $$H_{\mathbb{C}^*}(\mathbb{C},\mathfrak{g}_{\mathbb{C}})_{\mathbb{R}}\hookrightarrow H^{\infty}(S^1,\mathfrak{g})$$
This shows that there are no problems concerning the monodromy. So we have only to check the regularity aspect. 
For $g\in MG$ and $u\in M\mathfrak{g}$, $g\cdot(u)=gug^{-1}-g'g^{-1}\in M \mathfrak{u}$. 
On the other hand, using the description in the proof of lemma~\ref{cinfty}, we get for $u\in M\mathfrak{g}$ a transformation function $g(t):= \exp(tX)g(0)h(t)^{-1}$. A priori this function is in $L(S^1,G)$; but $\exp(tX)$ can be continued to a holomorphic function on $\mathbb{C}$, $g(0)$ is a constant and $h(t)^{-1}$ is a solution of the differential equation: $h'(t)=u(t)h(t)$; if $u(t)$ is defined on $\mathbb{C}^*$, this equation has a solution on the universal cover of $\mathbb{C}^*$, that is $\mathbb{C}$. So $g(t)$ is defined on $\mathbb{C}$, but has perhaps nontrivial monodromy; this is, of course, not possible, as the embedding tells us that $g(t)\subset L_{an}G$. 
\item {\bf The case of holomorphic loops on $A_n$}~
This case is exactly similar. $\mathbb{C}$ is replaced by $A'$ (compare subsection~\ref{A_n}).
\end{enumerate}
Hence theorem~\ref{polaractiononmg} is proved.
\end{proof}

Thus we have proven that $\sigma$-actions and diagonal actions are polar. Those two cases corresponds to the isotropy representation of Kac-Moody symmetric spaces of types $II$ and $IV$: the diagonal action is induced by the isotropy representation of Kac-Moody symmetric spaces in the non-twisted case, the $\sigma$-action is induced by the isotropy representation of $\sigma$\ndash twisted one.

The isotropy representations of Kac-Moody symmetric spaces of type $I$ and $III$ correspond to the Hermann examples. 
A holomorphic version of the Hermann examples~\cite{HPTT} can be defined in exactly the same way: 

Let $XG_{\mathbb{R}}^{\sigma}$ be a simply connected loop group, $\rho$ an involution such that $X\mathfrak{g}_{\mathbb{R}}^{\sigma}=\mathcal{K}\oplus \mathcal{P}$ is the decomposition into the $\pm 1$-eigenspaces of the involution induced by $\rho$ on $X\mathfrak{g}_{\mathbb{R}}^{\sigma}$. Let $XK_{\mathbb{R}}\subset XG_{\mathbb{R}}^{\sigma}$ be the subgroup fixed by $\rho$.  

\begin{theorem}
\label{polaractiononp}

The gauge action of $XK_{\mathbb{R}}$ on $\mathcal{P}$ is polar.
\end{theorem}

\begin{proof}
The proof is like the one of theorem~\ref{polaractiononmg}. One starts with a similar result for polar actions on Hilbert action~\cite{Terng95} and checks then step by step that the introduced higher regularity conditions fit together. 
\end{proof}

\section{Tame structures  on Kac-Moody algebras}
\label{holomorphicstructuresonkacmoodyalgebras}

In this section we describe explicit realizations as central extensions of holomorphic loop algebras of the abstract affine geometric Kac-Moody algebras which we introduced in definition~\ref{geometricaffinekacmoodyalgebra}.

\noindent Let $X\in \{A_n, \mathbb{C}^*\}$. As usual, holomorphic functions on $A_n$ are understood to be holomorphic in an open set containing $A_n$.

\begin{definition}[holomorphic affine geometric Kac-Moody algebra]
Define $\widehat{X\mathfrak{g}}$ to be an explicit realization  of $\widehat{L}(\mathfrak{g}, \sigma)$ with $L(\mathfrak{g}, \sigma)$ in the category of holomorphic maps $X\mathfrak{g}^{\sigma}$.
\end{definition}

Thus an element of a Kac-Moody algebra can be represented by a triple
$\left(f(z), r_c, r_d \right)$, where $f(z)$ denotes a $\mathfrak{g}_{\mathbb{C}}$-valued holomorphic function on $X$ and  $\{r_c, r_d\}\in \mathbb{C}$.

We equip those algebras with the norms $\|(f,r_c, r_d)\|_n:= \sup_{z\in A_n} |f_z| +(r_c \bar{r}_c+r_d \bar{r}_d)^{\frac{1}{2}}$. Thus we use the supremum norm on the loop algebra and complete it with an Euclidean norm on the double extension defined by $c$ and $d$.

\begin{lemma}[Banach- and Fr\'echet structures on Kac-Moody algebras]~
\label{banachandfrechetstructureonkacmoodyalgebras}
\begin{enumerate}
	\item For each $n$, the algebras $\widehat{A_n\mathfrak{g}}_{\mathbb{R}}^{\sigma}$ and $\widehat{A_n\mathfrak{g}}_{\mathbb{C}}^{\sigma}$ equipped with the norm $\|\hspace{3pt} \|_n$ are Banach-Lie algebras,
	\item $\widehat{M\mathfrak{g}}_{\mathbb{R}}^{\sigma}$ and $\widehat{M\mathfrak{g}}_{\mathbb{C}}^{\sigma}$ equipped with the sequence of norms $\|\hspace{3pt} \|_n$ are tame Fr\'echet-Lie algebras. 
\end{enumerate}
\end{lemma}

\begin{proof}
Let $\mathbb{F}\in \{\mathbb{R}, \mathbb{C}\}$.
\begin{enumerate}
\item As a consequence of lemma~\ref{mgfrechetagbanach}, $A_n\mathfrak{g}^{\sigma}$ is a Banach space. Thus $\widehat{A_n\mathfrak{g}}^{\sigma}$ is Banach.

To prove that $ad(f+r_cc+r_dd)$ is continuous, we use  $[f+r_cc+r_dd,g+s_cc+s_dd]= [f,g]+r_d[d,g]-s_d[d,f]=[f,g]_0+\omega(f,g)c+r_d izg'-s_diz f'$. 
Continuity follows from the continuity of $\frac{d}{dz}$, which is a consequence of the Cauchy-inequality and the boundedness of multiplication on compact domains.

\item $M\mathfrak{g}^{\sigma}$ is a tame Fr\'echet space as a consequence of lemma~\ref{mgfrechetagbanach}.  Thus $\widehat{M\mathfrak{g}}^{\sigma}$ is tame as a direct product of tame spaces (lemma~\ref{constructionoftamespaces}). To prove tameness of the adjoint action, we need tame estimates for the norms. Those estimates follow directly from the Banach space situation:
\begin{alignat*}{2}
&\|ad(f+r_cc+r_dd)(g+s_cc+s_dd)\|_n =\\
&= \|[f,g]_0+\omega(f,g)c+r_d izg'-s_diz f'\| \leq\\
&\leq 2\|f\|_n \sup_{z\in A_n}|g(z)|+\|\omega(f,g)c\|+\|r_d izg'\|+\|s_diz f'\|_n\leq\\
&\leq 2\|f\|_n \|g(z)\|_n+2\pi \|f\|_n \| g' \|_n+|r_d| \|z\|_n\|g'\|_n+|s_d| \|z\|_n\| f'\|_n\leq\\
&\leq  2\|f\|_n \|g\|_n + 2\pi\|f\|_n \frac{e^{n+1}}{e-1}  \|g\|_{n+1}+e^n\|f\|_n\|\frac{e^{n+1}}{e-1} \|g\|_{n+1} + e^n \|s_d\|\frac{e^{n+1}}{e-1} \| f\|_{n+1} \leq\\
&\leq\left(2+2\pi\frac{e^{n+1}}{e-1}+ 2\frac{e^{2n+1}}{e-1}  \right) \|g\|_{n+1}\|f\|_{n+1}\leq\\
&\leq 6\pi e^{2n+1}\|g\|_{n+1}\|f\|_{n+1}\,.
\end{alignat*}

Thus $ad(\widehat{g})$ is $(1, 0, 6\pi e^{2n+1}\|g\|_{n+1})$-tame.\qedhere
\end{enumerate}
\end{proof}

 This result shows that the tame structure on the Kac-Moody algebra is preserved by the adjoint action. 
 For additional analytic details and the Cauchy-inequalities see for example~\cite{Berenstein91}.

\section{Affine Kac-Moody groups}
\label{Kac-Moody groups}

\subsection{The loop group construction of affine Kac-Moody groups}

In this section we describe the construction of affine Kac-Moody groups as $2$-dimensional extensions of loop groups. Our presentation of the central extensions follows a construction proposed~\cite{PressleySegal86} for a Hilbert space setting; we use the tame Fr\'echet setting, developed in section~\ref{chap:tame}. Furthermore using a technical result of B.\ Popescu we prove that Kac-Moody groups of holomorphic loops carry a structure as tame Fr\'echet manifolds.

Let $G_{\mathbb{C}}$ denote a complex semisimple Lie group and $G$ its compact real form. As the constructions are valid in Kac-Moody groups defined with respect to various different regularity conditions, we use the regularity-independent notation $L(G_{\mathbb{C}}, \sigma)$ for the complex loop group and $L(G, \sigma)$ for its real form of compact type. To define groups of polynomial or analytic loops, we use the fact that every compact Lie group is isomorphic to a subgroup of some unitary group. Hence we can identify it with a matrix group. Similarly the complexification can be identified with a subgroup of some general linear group~\cite{PressleySegal86}. Groups of polynomial loops are then defined with respect to this representation. 

\noindent Kac-Moody groups are constructed in two steps.
\begin{enumerate}
\item The first step consists in the construction of an $S^1$-bundle in the real case (resp.\ a $\mathbb{C}^*$-bundle in the complex case) that corresponds via the exponential map to the central term $\mathbb{R}c$ (resp.\ $\mathbb{C}c$) of the Kac-Moody algebra.
\item In the second step we construct a semidirect product with $S^1$ (resp.\ $\mathbb{C}^*$). This corresponds via the exponential map to the $\mathbb{R}d$-\/ (resp.\ $\mathbb{C}d$-) term
\end{enumerate}

Study first the extension of $L(G, \sigma)$ with the short exact sequence:
\begin{displaymath}
1 \longrightarrow S^1\longrightarrow \widetilde{L}(G,\sigma) \longrightarrow L(G, \sigma) \longrightarrow 1\,.
\end{displaymath}

There are various groups $X$ that fit into this sequence. We need to define $\widetilde{L}(G,\sigma)$ in a way that its tangential Lie algebra at the identity $e\in \widetilde{L}(G, \sigma)$ is isomorphic to $\widetilde{L}(\mathfrak{g}, \sigma)$.

As described in~\cite{PressleySegal86} this $S^1$-bundle is best represented by triples: take triples $(g(z), p(z,t), w)$ where $g(z)$ is an element in the loop group, $p(z,t)$ a path connecting the identity to $g(z)$ and $w\in S^1$  (respective $w\in \mathbb C^*$) subject to the relation of equivalence: $(g_1(z), p_1(z,t), w_1) \sim (g_2(z), p_2(z,t), w_2)$ iff $g_1(z)= g_2(z)$ and $w_1= C_{\omega}(p_2*p_1^{-1})w_2$. The term $w_1= C_{\omega}(p_2*p_1^{-1})w_2$ defines a twist of the bundle. Here we put: 
\begin{displaymath}C_{\omega}(p_2*p_1^{-1})=e^{\int_{S\left(p_2*p_1^{-1}\right)}\omega}\end{displaymath}
 where $S(p_2*p_1^{-1})$ is a surface bounded by the closed curve $p_2*p_1^{-1}$ and $\omega$ denotes the $2$-form used to define the central extension of $L(\mathfrak{g}, \sigma)$. Group multiplication is defined by 
\begin{displaymath}(g_1(z), p_1(z,t), w_1)\cdot (g_2(z), p_2(z,t), w_2)=(g_1(z)g_2(z),\ p_1(z,t)*g_1(z)\cdot p_2(z,t), w_1 w_2)\,.\end{displaymath}

If $G$ is simply connected it can be shown that this object is a well defined group independent of arbitrary choices made in the construction iff $\omega$ is integral. This condition is satisfied by our definition of $\omega$~\cite{PressleySegal86}, theorem 4.4.1.
If $G$ is not simply connected, the situation is a little more complicated: let $G=H/Z$ where $H$ is a simply connected Lie group and $Z=\pi_1(G)$. Let $(LG)_0$ denote the identity component of $LG$. We can describe the  extension using the short exact sequence:
\begin{displaymath}1\longrightarrow S^1\longrightarrow \widetilde{LH}/Z\longrightarrow (LG)_0\longrightarrow 1\end{displaymath}see~\cite{PressleySegal86}, section 4.6.\, .

\noindent In case of complex loop groups, the $S^1$-bundle is replaced by a $\mathbb{C}^*$-bundle.

Hence, we can now give the definition of Kac-Moody groups:

\begin{definition}[Kac-Moody group]~
\begin{enumerate}
\item The real Kac-Moody group $\widehat{MG}_\mathbb{R}$ is the semidirect product of
 $S^1$ with the $S^1$-bundle $\widetilde{MG}_{\mathbb R}$.

\item The complex Kac-Moody group $\widehat{MG}_{\mathbb C}$ is its complexification: a semidirect product
of $\mathbb C^*$ with $\widetilde{MG}_{\mathbb C}$-bundle over
$MG$. 
\end{enumerate}
The action of the semidirect $S^1$ (resp.\ $\mathbb{C}^*$)-factor is in both cases defined by a shift of
the argument: 
\begin{displaymath}
\mathbb C^* \ni w_d: MG \rightarrow MG: f(z)\mapsto f(zw_d)\, .
\end{displaymath}
\end{definition}

\begin{remark}
Remark that in the compact case the shift is by elements $w_d=e^{i\varphi_d}$ only. Hence the action preserves the unit circle $S^1$. Thus one can use function spaces on $S^1$ thus yielding more general Kac-Moody groups than the groups of holomorphic loops, we study. Nevertheless those groups have no complexification in the same regularity class.
\end{remark}

The next aim is to prove that Kac-Moody groups are tame Fr\'echet manifolds. To this end we use a result of B.\ Popescu~\cite{Popescu05} stating that fiber bundles whose fiber is a Banach space over tame Fr\'echet manifolds are tame.
  
We start with the definition of tame fiber bundles:

\begin{definition}[tame Fr\'echet fiber bundle]
A fiber bundle $P$ over $M$ with fiber $G$ is a tame Fr\'echet manifold $P$ together with a projection map $\pi: P\longrightarrow M$ satisfying the following condition:

For each point $x\in M$ there is a chart $\varphi: U\longrightarrow V\subset F$ with values in a tame Fr\'echet space $F$ such that there is a chart $\varphi: \pi^{-1}(U)\longrightarrow  G\times U \subset G \times F$ such that the projection $\pi$ corresponds to a projection of $U\times F$ onto $U$ in each fiber.
\end{definition}

\noindent The following lemma is proved in~\cite{Popescu05}.
\begin{lemma}
Let $P$ be a fiber bundle over $M$ whose fiber is a Banach manifold; then $P$ is a tame Fr\'echet manifold.
\end{lemma}

\noindent This result contains the important corollary:

\begin{corollary}
The Kac-Moody groups $\widehat{MG}_{\mathbb{R}}$, $\widehat{MG}_{\mathbb{C}}$, the quotient spaces $\widehat{MG}_{\mathbb{C}}/\widehat{MG}_{\mathbb{R}}$, $\widehat{MG}_D$ and the quotient spaces $\widehat{MG}_{\mathbb{R}}/\textrm{Fix}(\rho)$ and $\widehat{MG}_{D}/\textrm{Fix}(\rho)$ are tame Fr\'echet manifolds.
\end{corollary}

\noindent Next we prove:

\begin{theorem}
\label{hatdiffeovectorspace}
$\widehat{MG}_{\mathbb{C}}/\widehat{MG}_{\mathbb{R}}$ is diffeomorphic to a vector space.
\end{theorem}

\begin{proof}
By theorem~\ref{mgc/mgrtamefrechet}, we know that $MG_{\mathbb{C}}/MG_{\mathbb{R}}$ is diffeomorphic to a vector space.  As $MG_{\mathbb{R}}$ is a subgroup of $MG_{\mathbb{C}}$, the quotient is well defined. To prove the theorem we check the decomposition
\begin{displaymath}\widehat{MG}_{\mathbb{C}}/\widehat{MG}_{\mathbb{R}}\simeq MG_{\mathbb{C}}/MG_{\mathbb{R}}\times (\mathbb{R^+})^2\,.\end{displaymath}

\noindent To this end we use the description of the elements in $\widehat{MG}_{\mathbb{C}}/\widehat{MG}$ as $4$-tuples. Two $4$-tuples $(g(z), p(z,t), r_c, r_d)$ and $(g'(z), p'(z,t), r'_c, r'_d)$ describe the same element of $\widehat{MG}_{\mathbb{C}}/\widehat{MG}$ iff there exists an element $(h(z), q(z,t), s_c, s_d) \in MG_{\mathbb{R}}$ such that 
\begin{displaymath}(g(z), p(z,t), r_c, r_d)=(g'(z)+h(z), p'(z,t)+q(z,t), r'_c+s_c, r'_d+s_d)\,.\end{displaymath}

Hence the equivalence classes for $g(z)$ are elements of $MG_{\mathbb{C}}/MG_{\mathbb{R}}$. The extension of $MG_{\mathbb{R}}$ lies in $S^1$. Thus $r_c$ and $r_d$ are defined up to an element in $S^1$. So we obtain a description of $\widehat{MG}_{\mathbb{C}}/\widehat{MG}_{\mathbb{R}}$ as a $(\mathbb{R}^+)^2$-bundle over $MG_{\mathbb{C}}/MG_{\mathbb{R}}$. 

As $MG_{\mathbb{C}}/MG_{\mathbb{R}}$ is diffeomorphic to a vector space (and hence simply connected), this bundle is trivial. The diffeomorphism can be described by mapping  $(g(z), p(z,t), r_c, r_d)$ onto $(g(z), r_c, r_d)$, hence by forgetting the path $p(z,t)$.
\end{proof}

\noindent Now we investigate the quotients $\widehat{MG}/\widehat{\textrm{Fix}(\sigma)}$.

The group $\widehat{\textrm{Fix}(\sigma)}$ consists of elements $(g, p, r_c, r_d)\in \widehat{MG}$ such that
 $\{g, p\in \textrm{Fix}(\sigma), r_c\in \pm 1, r_d \in \pm 1\}$. 
  As $\{r_c, r_d\} \in \pm 1$, this is a covering with four leaves. The details of the argument follow the description found in~\cite{Popescu05} for the smooth $C^{\infty}$-setting.

\noindent Let $H \subset MG_{\mathbb{C}}$ be a real form of non-compact type.

\noindent The description of the space $H/\textrm{Fix}(\sigma) \subset MG_{\mathbb{C}}/\textrm{Fix}(\sigma)$ follows similarly. This yields a proof  of the following theorem:

\begin{theorem}
\label{typeIIIdiffeovectorspace}
The space $H/\textrm{Fix}(\sigma)$ is diffeomorphic to a vector space.
\end{theorem}

In section~\ref{loopgroups} we introduced the pointwise exponential function $\Mexp:\widehat{M\mathfrak{g}}\longrightarrow \widehat{MG}$.
This exponential function can be extended to an exponential function on the Kac-Moody algebra $\widehat{M\mathfrak{g}}$.
\begin{displaymath}
\widehat{\Mexp}:\widehat{M\mathfrak{g}}\longrightarrow \widehat{MG}
\end{displaymath} 
Via $\widehat{\Mexp}$ the central extension $\mathbb{C}c$ corresponds to the fiber of the $\mathbb{C}^*$-bundle and the
 $\mathbb{C}d$-term corresponds to the $\mathbb{C}^*$-factor of the semidirect product.

To study more precisely the properties of the exponential function, we introduce the notion of curves in a Kac-Moody group. As a Kac-Moody group $\widehat{MG}_{\mathbb{R}}$ is locally a (topologically) direct product of the loop group $MG_{\mathbb{R}}$ with $\mathbb R^2$, a path: $\widehat{\gamma}: (-\epsilon, \epsilon)\rightarrow \widehat{MG}$ is locally described by three components: $\widehat{\gamma}(t)=(\gamma(t), \gamma_c(t), \gamma_d(t))$,
with $\gamma(t)$ taking values in $MG_{\mathbb{R}}$ and $\gamma_c(t)$, $\gamma_d(t)$ taking values in $\mathbb R$. For every $z\in \mathbb C^*$, $\gamma$ defines a path $\gamma_z(t): (-\epsilon, \epsilon)\rightarrow G$ by setting $\gamma_z(t):= [\gamma(t)](z)$.
A path $\gamma: (-\epsilon, \epsilon) \rightarrow MG$, $t \rightarrow \gamma (t)$ is differentiable (respective smooth) iff the map $\delta: (-\epsilon, \epsilon) \times \mathbb C^* \rightarrow G_{\mathbb C}$, $(t, z) \rightarrow \delta(t,z)$ such that $\delta(t,z)=\gamma_z(t)$ is differentiable (respective smooth).

For the group $\widetilde{MG}$ we can describe the exponential function as follows:

\begin{proposition}
The exponential function $\widetilde{\Mexp}:\widetilde{M\mathfrak{g}}\longrightarrow \widetilde{MG}$ is defined by
\begin{displaymath}
\widetilde{\Mexp}(f+r_c c)=\left(\Mexp (f),(t\longrightarrow exp tu)|_0^1 ,r^{ir_c}\right)
\end{displaymath}
\end{proposition}

For the straightforward proof one has to check that this defines a $1$-parameter subgroup whose differential at the identity is $f+r_c c$.

\subsection{The Adjoint action}
\label{theadjointaction}

Similarly to finite dimensional simple Lie groups a Kac-Moody group admits an adjoint action on its Lie algebra:

\begin{example}[Adjoint action]
\label{adjointaction}

With $x=\{w,(g,p,z)\}$, the Adjoint action of $\widehat{MG}^{\sigma}$ on $\widehat{M}\mathfrak{g}^{\sigma}$ is described by the following formulae
\begin{alignat*}{1}
\textrm{Ad}(x)u&:= gw(u)g^{-1}+\langle g w(u)g^{-1},g'g^{-1}\rangle c\\
\textrm{Ad}(x)c&:= c\\
\textrm{Ad}(x)d&:= d - g'g^{-1}+ \textstyle\frac{1}{2}\langle g'g^{-1}, g'g^{-1}\rangle c\,.
\end{alignat*} 
Here $\omega(u)$ denotes the shift of the argument by $\omega$.
\end{example}

\noindent For the proof compare~\cite{HPTT}, \cite{PressleySegal86}, and \cite{Kac90}.

\begin{proof}~ 
\begin{itemize}
\item $c$ generates the center, thus  $\textrm{Ad}(x)c:= c$.
\item $\textrm{Ad}(x)u$ follows by integrating the $Ad$-action.
\item To calculate $\textrm{Ad}(g) (d)$, we use the $Ad$-invariance of the Lie bracket. As $[d,v]=v'$ we get for all $v \in L_{alg}\mathfrak{g}^{\sigma}$:
\begin{alignat*}{2}
  &   gv'g^{-1}+\langle g v'g^{-1},g'g^{-1}\rangle c =\\
= &   \textrm{Ad}(g)(v')=\textrm{Ad}(g) [d,v]= [\textrm{Ad}(g)(d), \textrm{Ad}(g) (v)]=\\
= &  [h+ \mu c + \nu d, gvg^{-1}+\langle g vg^{-1},g'g^{-1}\rangle c]=\\
= &  [h, gvg^{-1}]+\nu[d, gvg^{-1}]                           =\\
= &  hgvg^{-1}-gvg^{-1}h+\omega(h, (gvg^{-1})')c+\nu g'vg^{-1}+\nu gv'g^{-1}+\nu gv(g^{-1})'\,,
\end{alignat*}
 with $h \in L_{alg}\mathfrak{g}^{\sigma}$ and $\{\mu, \nu\} \in \mathbb{R}$.

\noindent To get equality we have to choose $\nu=1$, $h=-g'g^{-1}$. This gives us
 \begin{alignat*}{2} 
&  hgvg^{-1}-gvg^{-1}h+\omega(h, (gvg^{-1})')c+\nu g'vg^{-1}+\nu gv'g^{-1}+\nu gv(g^{-1})'=\\
=&-g'vg^{-1}+gvg^{-1}g'g^{-1}+\omega(-g'g^{-1}, (gvg^{-1})')c+ g'vg^{-1}+ gv'g^{-1}+ gv(g^{-1})'=\\
=&\omega(-g'g^{-1}, (gvg^{-1})')c+  gv'g^{-1}\,.
\end{alignat*}

\noindent Thus we are left with the calculation of $\mu$. To this end we use the property that $Ad$ acts by isometries. This results in $\mu=\frac{1}{2}\langle g'g^{-1}, g'g^{-1}\rangle$.\qedhere

\end{itemize}
\end{proof}

\noindent More details can be found in~\cite{PressleySegal86,HPTT,Popescu05, Popescu06}.

\begin{theorem}[Polarity of the Adjoint action]
\label{polarityofadjointaction}
Let $H_{l, r}\subset \widehat{X\mathfrak{g}}^{\sigma}$, $\{l, r\} \in \mathbb{R}\backslash \{0\}$ denote the intersection of the sphere with radius $-l^2$ with the horosphere $r_d=r$. The restriction of the Adjoint action to $H_{l, r}$ is polar.
\end{theorem}

\begin{proof} The restriction of the Adjoint action to $H_{l, r}$ coincides with the gauge action on $X\mathfrak{g}$. Hence, theorem \ref{polarityofadjointaction} is a direct consequence of theorem~\ref{polaractiononmg}.
\end{proof}

C.-L.\ Terng describes how to associate an affine Weyl group to this gauge action~\cite{Terng95}. This is exactly the affine Weyl group of the Kac-Moody group $\widehat{MG}$.

\noindent This theorem gives a complete description of the Adjoint action iff $r_d\not=0$. 

\noindent Surprisingly in the remaining case $r_d=0$ the situation is different:
now the Adjoint action is reduced to the equations:
 \begin{alignat*}{1}
\textrm{Ad}(x)u&:= gw(u)g^{-1}+\langle g w(u)g^{-1},g'g^{-1}\rangle c\\
\textrm{Ad}(x)c&:= c
\end{alignat*}

Calculate the orbit of the constant function $u\equiv0$. $u$ is fixed by the Adjoint action as $\textrm{Ad}(x)u:= g 0 g^{-1}+\langle g0g^{-1},g'g^{-1}\rangle c=0$. Hence iff we can  describe the restriction of the Adjoint action to $X\mathfrak{g}^{\sigma}$ as some kind of polar action then the associated Weyl group has to be necessarily of spherical type. Furthermore the action is clearly not proper Fredholm. Hence the Hilbert-space version is not covered by Terng's results. 

\noindent We use the regularity independent notation:

We define a  flat of finite type to be a flat $\mathfrak{t}\subset L(\mathfrak{g},\sigma)$ such that $\mathfrak{t}$ is the restriction of a flat in $\widehat{L}(\mathfrak{g},\sigma)$. Hence all flats of finite type are conjugate in $\widehat{L}(G,\sigma)$ and as the orbits of $\widehat{L}(G,\sigma)$ and $L(G,\sigma)$ coincide on $L(\mathfrak{g},\sigma)$ also in $L(G,\sigma)$.  Hence any flat of finite type in $L(\mathfrak{g}, \sigma)$ is isomorphic to $\mathfrak{t}_0\subset \mathfrak{g}$, where we choose $\mathfrak{g}$ to denote the subalgebra of constant loops. Using the usual notion for regular and singular elements we find that the associated Weyl group is the spherical Weyl group of $\mathfrak{g}$.

\begin{remark}
From a geometric point of view this different behavior is related to the fact that the hyperplane defined by $r_d=0$ corresponds to the spherical building at infinity while the space $r_d\not=0$ corresponds to the spherical building at infinity. For further details confer \cite{Freyn09}, \cite{Freyn10a}, \cite{Freyn10d}.
\end{remark}

\bibliographystyle{alpha}
\bibliography{Doktorarbeit1}

\end{document}